\documentclass[12pt,a4paper]{amsart}
\usepackage{amsfonts}
\usepackage{amssymb}
\usepackage{amsmath}
\usepackage{amsthm} 
\usepackage{cite}
\usepackage{enumitem}
\usepackage{tikz}
\usepackage{subfigure}
\usepackage{latexsym}
\usepackage{dsfont}

\theoremstyle{plain}
\newtheorem{thm}{Theorem}

\newtheorem{lem}{Lemma}
\newtheorem{prop}{Proposition}
\newtheorem{cor}{Corollary}
\newtheorem{rem}{Remark}

\newcommand{\vecII}[2]{
\ensuremath{
\begin{pmatrix}  
#1 \\ #2 \\
\end{pmatrix}}}

\newcommand{\bra}[1]{\langle #1 \rangle}
\providecommand{\ind}{\mathds{1}} 
\providecommand{\les}{\lesssim}
\providecommand{\sm}{\setminus}
\providecommand{\N}{\mathbb{N}}
\providecommand{\R}{\mathbb{R}}
\providecommand{\Z}{\mathbb{Z}}
\providecommand{\C}{\mathbb{C}}
\providecommand{\eps}{\varepsilon}
\providecommand{\ov}{\overline}

\DeclareMathOperator{\supp}{supp}
\DeclareMathOperator{\Real}{Re}
\DeclareMathOperator{\sign}{sign}
\DeclareMathOperator{\dist}{dist}

\DeclareMathOperator{\BMO}{BMO}
\renewcommand{\qed}{\hfill $\Box$}

\setitemize{itemsep=+2pt} 
\setenumerate{itemsep=+2pt}

\begin{document}

%\input{Changes}
%\newpage

\allowdisplaybreaks

\title{On Gagliardo-Nirenberg Inequalities with vanishing symbols}

\author{Rainer Mandel\textsuperscript{1}}
\address{\textsuperscript{1}Karlsruhe
Institute of Technology, Institute for Analysis, Englerstra{\ss}e 2, 76131 Karlsruhe, Germany}
\email{Rainer.Mandel@kit.edu}
  
\subjclass[2020]{35A23} 
\keywords{Gagliardo-Nirenberg Inequality, Interpolation} 

\begin{abstract}
  We prove interpolation inequalities of Gagliardo-Nirenberg type involving Fourier symbols that vanish on
  hypersurfaces in $\R^d$.
\end{abstract}

\maketitle
\allowdisplaybreaks

 \section{Introduction}

 In a recent paper by~Fern\'{a}ndez, Jeanjean, Mari\c{s} and the author
 the following inequality of Gagliardo-Nirenberg type was proved
 \begin{align}\label{eq:GN} 
   \|u\|_q \les \|(|D|^s-1)u\|_2^{1-\kappa}\|u\|_2^\kappa \qquad (u\in\mathcal S(\R^d)).
  \end{align}  
 Here, $(|D|^s-1)u= \mathcal F^{-1}((|\cdot|^s-1)\hat u)$, the symbol $\les$ stands for $\leq C$ for some
 positive number $C$ independent of $u$  and the parameters  are supposed to satisfy 
 \begin{equation}\label{eq:GNConditions}
   s>0,\,\kappa\geq \frac{1}{2},\,2\leq q<\infty,\, d\in\N,d\geq 2
   \quad\text{and}\quad   
   \frac{2(1-\kappa)}{d+1}\leq  \frac{1}{2}-\frac{1}{q} \leq \frac{(1-\kappa)s}{d}, 
 \end{equation} 
 see~\cite[Theorem~2.6]{FerJeaManMar}.  
 In this paper we investigate such inequalities in greater generality both by extending the analysis to a
 larger class  of exponents, but also by allowing for  more general Fourier symbols. We expect applications in
 the context of normalized solutions of elliptic PDEs and orbital
 stability~\cite{CazLio,BarJeaSoa,NorisTavaresVerzini} or long-time behaviour~\cite{Weinstein} of
 time-dependent PDEs just asin the case of the classical Gagliardo-Nirenberg
 Inequality~\cite{Nirenberg_OnElliptic}.
 % (see  \cite{Gagliardo_Ulteriori,Nirenberg_Extended} for the bounded domain case). 
  In~\cite{FerJeaManMar} and~\cite{LenzmannWeth} applications of~\eqref{eq:GN} to variational existence
  results and symmetry breaking phenomena for biharmonic nonlinear Schr\"odinger equations are given.
  For the   existence and qualitative properties of maximizers in classical Gagliardo-Nirenberg inequalities
  we refer to~\cite{Weinstein,DelPinoDolbeault,BellFraVis,LenzmannSok,Zhang}. Interpolation inequalities in different
  spaces like Lorentz spaces, Besov spaces, BMO or weighted Lebesgue spaces can be found
  in~\cite{BrezisVanSch,HajMolOzaWa_Necessary,BreMiro_SobolevGN,DaoLam,CaffKohnNir,McCorRobRod_GN}.
 
 \medskip
 
 We shall be concerned with inequalities of the form
\begin{align}\label{eq:GNgeneral}
    \|u\|_q \les \|P_1(D)u\|_{r_1}^{1-\kappa}\|P_2(D)u\|_{r_2}^\kappa
\end{align}
where $q,r_1,r_2\in [1,\infty],\kappa\in [0,1]$ and $P_1,P_2:\R^d\to\R$ are Fourier symbols that
may vanish on a given smooth compact hypersurface $S\subset\R^d,d\geq 2$ with at least $k\in\{1,\ldots,d-1\}$
non-vanishing principal curvatures in each point. In the case $d=1$ the symbols are allowed to have a finite
set of zeros $S\subset \R$. 
We will assume that $P_i$ vanishes of order $\alpha_i$
on $S$ and behaves like $|\cdot|^{s_i}$ at infinity, see Assumption~(A1),(A2) below for a precise statement.
This covers~\eqref{eq:GN} as a special case where $d\geq 2$, 
$(\alpha_1,\alpha_2,s_1,s_2)=(1,0,s,0)$ and $S$ is the unit sphere in $\R^d$, so
$k=d-1$.  As an application of our results for~\eqref{eq:GNgeneral} we obtain the following generalization of
\cite[Theorem~2.6]{FerJeaManMar}.
 
\begin{thm}\label{thm:GNhigherDspecial} 
  Assume $d\in\N,d\geq 2,\kappa\in [0,1],s>0$. Then 
  $$
     \|u\|_q \les \|(|D|^s-1)u\|_r^{1-\kappa}\|u\|_r^\kappa \qquad (u\in\mathcal S(\R^d)) 
  $$
  holds provided that the exponents $r\in [1,2], q\in [2,\infty]$ satisfy
  $$
    \frac{2(1-\kappa)}{d+1} \leq \frac{1}{r}-\frac{1}{q} \leq \frac{(1-\kappa)s}{d}  
    \quad\text{and}\quad
     \min\left\{\frac{1}{r},\frac{1}{q'}\right\} 
     \begin{cases}
       \geq \frac{d+1-2\kappa}{2d} &\text{if }\kappa>0, \\
       > \frac{d+1}{2d} &\text{if }\kappa=0.
     \end{cases}
  $$ 
\end{thm}

 So our result from~\cite{FerJeaManMar} is recovered as~\eqref{eq:GNConditions} is nothing
 but the special case $r=2$ in the above theorem. We  even obtain  
 sufficient conditions for general $q,r_1,r_2\in [1,\infty]$.
 In the one-dimensional case we obtain the following generalization of \cite[Theorem~2.3]{FerJeaManMar}.

\begin{thm}\label{thm:GN1Dspecial}
   Assume $\kappa\in [0,1],s>0$. Then  
  $$
     \|u\|_q \les \|(|D|^s-1)u\|_{r_1}^{1-\kappa} \|u\|_{r_2}^\kappa \qquad (u\in\mathcal S(\R)) 
  $$
  holds provided that $q,r_1,r_2\in [1,\infty]$ satisfy $1-\kappa \leq \frac{ 
  1-\kappa}{r_1}+\frac{\kappa}{r_2} -  \frac{1}{q} \leq (1-\kappa)s$.
  % and, if $0<s<1$, $(q,r_1,r_2)\neq (\infty,\frac{1}{s},\infty)$.   
\end{thm}

 Both our main results arise as special cases of Theorem~\ref{thm:GN1D} and Theorem~\ref{thm:GNhigherD} where
 interpolation inequalities of the form~\eqref{eq:GNgeneral} are proved for symbols $P_1,P_2:\R^d\to\R$ that
 satisfy the following abstract conditions:
  \begin{itemize}
  \item[(A1)] There is a compact hypersurface
  $S=\{\xi\in\R^d:F(\xi)=0\}$ with $F\in C^\infty(\R^d)$, $|\nabla F|\neq 0$ on $S$ and at least
  $k\in\{1,\ldots,d-1\}$ nonvanishing principal curvatures at  each point such that
  $\{\xi\in\R^d: P_i(\xi)=0\}\subset S$. For $\xi$ near $S$ we have
  $P_i(\xi)= a_{i+}(\xi)F(\xi)_+^{\alpha_i} + a_{i-}(\xi)F(\xi)_-^{\alpha_i}$ for smooth non-vanishing
  functions $a_{i+},a_{i-}$ and  $\alpha_i>-1$. In the case $\alpha_i=1$ additionally assume
  $a_{i-}=-a_{i-}$ and in the case $\alpha_i=0$ additionally assume $a_{i-}=a_{i+}$.
  \item[(A2)] 
%   ~ \r{ such that $|\partial^\alpha (1/P_i
%   (\xi))|\leq C(1+|\xi|)^{-s_i-|\alpha|}$ holds for $|\alpha|\leq \lceil\frac{d}{2}\rceil + 1$ outside of this
%   neighbourhood where $s_1,s_2\in\R$.} \\
    There are $s_1,s_2\in\R,\delta>0$ such that for $\dist(\xi,S)\geq \delta>0$ the functions
   $Q_i(\xi):= \langle\xi\rangle^{s_i}/ P_i(\xi)$ satisfy for some $\eps>0$
   \begin{align*}
    &\left| \partial^\gamma Q_i(\xi) \right| 
      \les \langle\xi\rangle^{-|\gamma|} 
      &&\hspace{-3cm}\text{ if }\gamma\in\N_0^d,\;\,  0\leq
      |\gamma|\leq \left\lfloor d/2\right\rfloor,\\    
    &\left| \partial^\gamma Q_i(\xi) \right| 
      \les \langle\xi\rangle^{-\eps-|\gamma|}
    &&\hspace{-3cm}\text{ if }\gamma\in\N_0^d,\;\,  
      |\gamma| = \left\lfloor d/2\right\rfloor + 1.
  \end{align*}
  \end{itemize}
  Here and in the following we set  $\langle\xi\rangle:= (1+|\xi|^2)^{1/2}$ and  
  $|\gamma|:=|(\gamma_1,\ldots,\gamma_d)|:=\gamma_1+\ldots+\gamma_d$  for multi-indices $\gamma\in\N_0^d$,
  $F(\xi)_+:=\max\{F(\xi),0\}$ and $F(\xi)_-=:-\min\{F(\xi),0\}$. In the case $d=1$ assumption (A1) is
  supposed to mean   $S=\{\xi\in\R: F(\xi)=0\}=\{\xi_1^*,\ldots,\xi_L^*\}$ with $F,P_i,a_{i+},a_{i-}$
  as above. Given the importance of the fractional Laplacian $(-\Delta)^{s/2}  = |D|^s$ we mention that one
  may generalize this further by   allowing the symbols $P_1,P_2$ to vanish at some finite set of points in
  $\R^d\sm S$, see Remark~\ref{rem:GNandFracLap}. The choice $P_1=P_2$ or $\kappa\in\{0,1\}$ leads to
  Sobolev inequalities. In the elliptic case  $-\Delta-1=|D|^2-1$ such results are due to  Kenig, Ruiz,
  Sogge~\cite[Theorem~2.3]{KeRuSo}, Guti\'{e}rrez~\cite[Theorem~6]{Gut} and Evequoz
  \cite{EveqPlane}. Our  most general result from  Theorem~\ref{thm:GNhigherD} contains these results as a
  special case $(k,s_1,\alpha_1,\kappa)=(d-1,2,1,0)$.
  Sharp results for special non-elliptic symbols with unbounded characteristic set $S$ are due to 
  Kenig, Ruiz, Sogge~\cite[Theorem~2.1]{KeRuSo}, Koch, Tataru \cite{KochTat_Principally} and
  Jeong, Kwon, Lee~\cite[Theorem~1.1]{JeongKwonLee_Uniform}. 
    
%We conjecture that admits a generalization that covers the whole range
%$\ov\alpha\in [0,\frac{k+2}{2}]$.

%    Our approach can be extended to the case
%   $k=0$ of a locally flat hypersurface $S\subset\R^d$ up to introducing further case dinstinctions,  which we
%   prefer to avoid here.
   
  \begin{rem} ~
  \begin{itemize}
    \item[(a)] In the case $S=\emptyset$ the main results of this paper hold without any assumption on
    $\alpha_1,\alpha_2$. %In fact we then have $u_2=u$ in Proposition~\ref{prop:LargeFreqInterpolated}.
    Similarly, if the Fourier support of the given functions is contained in a fixed compact subset of $\R^d$,
    then all conditions involving $s_1,s_2$ can be ignored.
    \item[(b)] Theorem~\ref{thm:GNhigherDspecial} and \ref{thm:GN1Dspecial} equally hold for
    symbols $P_i(|D|)$ where $P_i$ are polynomials of degree $s$ with simple zeros only or no zeros at all.
    \item[(c)] Our analysis may be extended to  vectorial differential operators with constant
    coefficients $P_1(D),P_2(D)$ where, according to Cramer's rule, the characteristic set $S$ is then
    supposed to satisfy $\{\det(P_i(\xi))=0\}\subset S$ for $i=1,2$.
    Such a situation  occurs in the context of Maxwell's equations, Dirac equations or Lam\'{e}
    equations with constant coefficients.  
    \item[(d)] The Gagliardo-Nirenberg inequalities from this paper hold for functions
    with Fourier support in bounded smooth pieces of more general sets $S\subset\R^d$. In this way,  
    unbounded characteristic sets $S$ or characteristic sets with singularities as in
    \cite[Section~3]{ManSch} may be partially analyzed, but a full analysis remains to be done.  In the
    special case of the wave and Schr\"odinger operator one may nevertheless implement the strategy
    from~\cite{FerJeaManMar} to get such inequalities at least for $r=2$, see
    Section~\ref{sec:GN_wave}.
    \item[(e)] The admissible set of exponents for Gagliardo-Nirenberg inequalities may become larger
    by imposing   symmetries. For instance, the Stein-Tomas Theorem for
    $O(d-k)\times O(k)$-symmetric functions from \cite{ManDOS} may substitute the
    classical Stein-Tomas Theorem in Lemma~\ref{lem:Tdelta2} to prove better dyadic estimates. The latter
    yield larger values for $A_\eps(p,q)$ in~\eqref{eq:def_Aeps}, which allows to deduce  Gagliardo-Nirenberg
    inequalities for a wider range of exponents.
  \end{itemize}
  \end{rem}

  \medskip

  Our strategy is as follows. We decompose the pseudo-differential operators
  $P_1(D),P_2(D)$ dyadically, both for frequencies close to the critical surface $S$ and at infinity.
  Assumption (A1) allows to analyze the first-mentioned part with the aid of Bochner-Riesz estimates
  from~\cite{ManSch,ChoKimLeeShim2005}. Here, only the parameters $\alpha_1,\alpha_2$ will play a
  role. Assumption (A2) will be used to estimate the second-mentioned part that only involves $s_1,s_2$. 
  Interpolating the bounds for the dyadic operators in both frequency regimes then allows to conclude. We
  stress that the proof from~\cite{FerJeaManMar} does not carry over from the $L^2(\R^d)$-setting since Plancherel's Theorem does not have a counterpart in $L^r(\R^d)$
  with $r\neq 2$.

% \r{Notation: $L^p(\R^d)$ with $\|\cdot\|_p,\mathcal F f,\hat f$ Fourier transform of $f$, $\les$}

\section{Preliminaries}

In the following we  decompose a given Schwartz function $u\in\mathcal S(\R^d)$ in frequency space. We  start
by separating the frequencies close to the critical surface from the others by defining
\begin{equation}\label{eq:u1u2}
  u_1:= \mathcal F^{-1}(\tau \hat u),\;
  u_2:= \mathcal F^{-1}( (1-\tau) \hat u)
  \quad\text{where } \tau\in C_0^\infty(\R^d),\;
   \tau =1\text{ near }S.
\end{equation}
More precisely, $\tau$ is chosen in such a way that $S$ admits local parametrizations in Euclidean
coordinates within $\supp(\tau)$, that $a_{i+},a_{i-}$ from (A1) are uniformly positive near $S$ and that
the functions $Q_i$ from (A2) behave as required for $\xi\in\R^d\sm\supp(\tau)$. The function $\tau$ is
considered as fixed from now on.
For both $u_1$ and $u_2$ we will introduce a dyadic decomposition into infinitely many annular regions in
order to prove our estimates mostly via Bourgain's summation argument~\cite{Bourgain1985}. We will need the
following abstract version of this result from~\cite[p.604]{CarberySeegerWaingerWright1999}.

\begin{lem}\label{lem:SummationLemma} 
Let $\beta_1,\beta_2\in\R,\theta\in (0,1)$, let $(X_1,X_2)$ and $(Y_1,Y_2)$ be real interpolation pairs of
Banach spaces. For $j \in \mathbb{N}$ let $\mathcal T_j$ be linear operators satisfying
\begin{align*}
\| \mathcal T_jf \|_{Y_1} \leq M_1\, 2^{\beta_1 j}  \| f \|_{X_1}, \qquad
\| \mathcal T_jf \|_{Y_2} \leq M_2\, 2^{\beta_2 j} \| f \|_{X_2}.
\end{align*}
Then we have
\begin{align}
\label{eq:SummationI}
  \| \sum_{j\in\N} \mathcal T_jf \|_{(Y_1,Y_2)_{\theta,\infty}}  
  \leq C(\beta_1,\beta_2) M_1^{1-\theta} M_2^{\theta} \|f\|_{(X_1,X_2)_{\theta,1}} 
\end{align}
 provided that $(1-\theta)\beta_1+\theta\beta_2=0$ with $\beta_1,\beta_2\neq 0$.  
 In the case $(1-\theta)\beta_1+\theta\beta_2<0$ we have for all $r\in [1,\infty]$
\begin{align}
\label{eq:SummationII}
  \| \sum_{j\in\N} \mathcal T_jf \|_{(Y_1,Y_2)_{\theta,r}}  
  \leq C M_1^{1-\theta} M_2^{\theta} \|f\|_{(X_1,X_2)_{\theta,r}}. 
\end{align}
\end{lem}

The whole point of this result is \eqref{eq:SummationI}; the estimate~\eqref{eq:SummationII} is a
rather trivial consequence of the summability  of the interpolated bounds 
$$
  \|\mathcal T_jf \|_{(Y_1,Y_2)_{\theta,r}} 
  \les 2^{j((1-\theta)\beta_1+\theta\beta_2)} \|f\|_{(X_1,X_2)_{\theta,r}}
  \quad\text{for all }r\in [1,\infty].
$$
Here, $(Y_1,Y_2)_{\theta,r},(X_1,X_2)_{\theta,r}$ denote real interpolation spaces~\cite{BerghLoefstrom1976}. 
The choice $Y_1=L^{q_1}(\R^d),Y_2=L^{q_2}(\R^d)$  with 
$\frac{1}{q}=\frac{1-\theta}{q_1}+\frac{\theta}{q_2},q_1\neq q_2$ yields  the Lorentz space 
$(Y_1,Y_2)_{\theta,r}=L^{q,r}(\R^d)$ whereas   $q_1=q_2=q$  leads to
$(Y_1,Y_2)_{\theta,r}=L^q(\R^d)$.
In our context, the spaces $X_i$ are defined as the completion of $\{u\in \mathcal S(\R^d) : P_i(D)u\in
L^r(\R^d)\}$ with respect to the norm $\|u\|_{X_i}:=\|P_i(D)u\|_r$. Exploiting assumption (A1),(A2) we
find that for any given $u\in\mathcal S(\R^d)$ the function $P_i(D)u$ is a priori well-defined as a function
in $L^\infty(\R^d)$ because $\xi\mapsto P_i(\xi)\hat u(\xi)$ is integrable due to $\alpha_i>-1$.
(Choosing the completion of a smaller set one may extend the analysis to $\alpha_i\leq -1$.) 
The link to Gagliardo-Nirenberg-type inequalities is provided by the general
interpolation property \cite[Theorem~3.1.2]{BerghLoefstrom1976}, namely
\begin{equation*}%\label{eq:interpolationfunctor}
   \|f\|_{(X_1,X_2)_{\kappa,r}}
   %\leq \|f\|_{(X_1,X_2)_{\kappa,1}}
   \leq \|f\|_{X_1}^{1-\kappa} \|f\|_{X_2}^{\kappa}
   \qquad (0<\kappa<1,1\leq r\leq \infty).
\end{equation*}
In fact, choosing $X_1,X_2$ as above we obtain for $u\in\mathcal S(\R^d)$
\begin{equation}\label{eq:interpolationfunctor}
  \|u\|_{(X_1,X_2)_{\kappa,r}}
   \leq \|P_1(D)u\|_{r_1}^{1-\kappa} \|P_2(D)u\|_{r_2}^{\kappa}
   \qquad (0<\kappa<1,1\leq r\leq \infty).
\end{equation}
The same estimate holds for $(X_1,X_2)_{\kappa,r}$ replaced by the complex interpolation space
$[X_1,X_2]_{\kappa}$. This can be deduced from  \eqref{eq:interpolationfunctor} and~$[X_1,X_2]_{\kappa}\subset
(X_1,X_2)_{\kappa,\infty}$, see \cite[Theorem~4.7.1]{BerghLoefstrom1976}.

\section{Large frequency analysis}

We start with our analysis for large frequencies or, more precisely, for those frequencies with
uniformly positive distance to the critical surface $S$ given by our assumption (A1). To this end we first
choose a function $\eta$ such that
  \begin{equation*}%\label{eq:choiceEta}
    \eta\in C_0^\infty(\R),\quad  \supp(\eta)\subset [-2,-\frac{1}{2}]\cup [\frac{1}{2},2],\quad
    \sum_{j\in\Z} \eta(2^j\cdot)=1 \text{ almost everywhere on } \R,
  \end{equation*}
  see~\cite[Lemma~6.1.7]{BerghLoefstrom1976}. For $\xi_0\in\R^d$  define  
  \begin{align} \label{eq:def_Tj}
    \begin{aligned}
    T_j f
    &:= \mathcal F^{-1}\left( \eta(2^j|\xi-\xi_0|) \hat f \right)
    = K_j\ast f \qquad \text{where}\\
    K_j(x)
    &:= \mathcal F^{-1}\left( \eta(2^j|\xi-\xi_0|)\right)(x)
     = 2^{-jd}  \mathcal F^{-1}\left( \eta(|\cdot|)\right)(2^{-j}x)e^{ix\cdot\xi_0}.
    \end{aligned}
  \end{align}
  Later on, we will choose $\xi_0\in S$ in order to have 
  $T_ju_2=0$ for $j\geq j_0$ where $j_0\in\Z$ only depends on $\xi_0$ and $\tau$.
  Indeed, \eqref{eq:u1u2} implies that $\hat u_2(\xi) = (1-\tau(\xi))\hat u(\xi)$ vanishes for
  frequencies $\xi$ close to $S$. As a consequence, only the bounds for $j\searrow -\infty$ will be of
  importance.

\begin{lem}\label{lem:Tdelta}
  Assume $d\in \N$ and let $\eta\in C_0^\infty(\R)$, $\xi_0\in \R^d$. Then we have for $j\in\Z$
  $$
    \|T_j\|_{p\to q} \les 2^{-jd(\frac{1}{p}-\frac{1}{q})}  
     \qquad \text{for }1\leq p\leq q\leq \infty.
  $$
\end{lem}
\begin{proof}
  For all $r\in [1,\infty]$ we have $\|K_j\|_r
    =  2^{-jd}  \|\mathcal F^{-1}\left( \eta(|\cdot|)\right)(2^{-j}\cdot)\|_r
    \les  2^{-j\frac{d}{r'}}$.
  Hence, for $1\leq p\leq q\leq \infty$ and $\frac{1}{r}:=1+\frac{1}{q}-\frac{1}{p}$ we get from Young's
  Convolution Inequality 
  $$
    \|T_j f\|_q
    \les \|K_j\|_r \|f\|_p
    \les  2^{-j\frac{d}{r'}} \|f\|_p
    \les  2^{-jd(\frac{1}{p}-\frac{1}{q})} \|f\|_p.   
  $$
\end{proof}

 In the following, we will need a multiplier theorem in $L^\mu(\R^d)$ for arbitrary $\mu\in
[1,\infty]$. The natural candidate -  Mikhlin's multiplier theorem
\cite[Theorem~6.1.6]{BerghLoefstrom1976} - is only available for $\mu\in (1,\infty)$.
In order to avoid tiresome separate discussions  we first provide a simple sufficient condition for a given
function $m:\R^d\to \R$ to  be a $L^\mu$-multiplier for all $\mu\in [1,\infty]$.
The following result essentially says that a function $m$ serves our purpose provided that its derivatives
grow a bit slower near zero and decay a bit faster near infinity compared to the requirements of Mikhlin's
multiplier theorem.

\begin{prop}\label{prop:multiplier}
  Let $d\in\N, k:= \lfloor d/2\rfloor +1$ and $m\in C^k(\R^d\sm\{0\})$. Then $m$ is an
  $L^\mu$ multiplier for all $\mu\in [1,\infty]$ provided that there is $\eps>0$
  such that
  $$
    |\partial^\alpha m(\xi)| \les \langle \xi\rangle^{-2\eps} |\xi|^{-k+\eps}
    \text{ for all }\alpha\in\N_0^d \text{ such that }|\alpha|=k.
  $$
\end{prop}
\begin{proof}
  We show that the assumptions imply that $\rho:= \mathcal F^{-1}m$ is integrable. Once this is shown, the
  result follows from Young's Convolution Inequality because of
  $$
    \|\mathcal F^{-1}(m\hat f)\|_\mu
    =  \|\rho\ast   f\|_\mu
    \leq \|\rho\|_1 \| f\|_\mu.
  $$
  We may  w.l.o.g. assume
  $0<\eps\leq 2k-d$. For all $\alpha\in\N_0^d, |\alpha|=k$ we have
  $$
    |\mathcal F\left( (-ix)^\alpha\rho\right)(\xi)|
    = |\partial^\alpha \hat \rho(\xi)|
    = |\partial^\alpha m(\xi)|
    \les \langle \xi\rangle^{-2\eps} |\xi|^{-k+\eps}.
  $$
  Hence, $\mathcal F(x^\alpha\rho)$ belongs to the space $L^{\sigma_1}(\R^d)\cap L^{\sigma_2}(\R^d)$ where
  $\sigma_1:=\frac{d}{k+\eps/2},\sigma_2:=\frac{d}{k-\eps/2}$. Our choice for $\eps$ implies
  $1\leq \sigma_1\leq \sigma_2\leq 2$, so the Hausdorff-Young Inequality gives
  $$
    |x|^k \rho \in L^{\sigma_1'}(\R^d)\cap L^{\sigma_2'}(\R^d).
  $$
  To conclude $\rho\in L^1(\R^d)$ with H\"older's Inequality it remains to check
  $$
    |x|^{-k} \in L^{\sigma_1}(\R^d)+L^{\sigma_2}(\R^d).
  $$
  But this follows from $|x|^{-k}\ind_{|x|\leq 1} \in L^{\sigma_1}(\R^d)$ and $|x|^{-k}\ind_{|x|>1}\in
  L^{\sigma_2}(\R^d)$ due to $k\sigma_1<d<k\sigma_2$, which finishes the proof.
\end{proof}

  Next we provide our estimates in the large frequency regime. To this end we analyze the mapping properties 
  of $\mathcal T_j u:= T_j(u_2)$ where $T_j$ and $u_2=\mathcal F^{-1}((1-\tau)\hat u)$ were defined in
  \eqref{eq:def_Tj},\eqref{eq:u1u2}, respectively.

  \begin{prop} \label{prop:dyadicEstimatesI}
    Assume $d\in\N$ and (A2) with $s_1,s_2\in\R$. Then, for $i=1,2$, 
     $$
       \|\mathcal T_ju\|_{q} \les  2^{j(s_i-d(\frac{1}{p}-\frac{1}{q}))} \|P_i(D)u\|_p.
        \qquad \text{for }1\leq p\leq  q\leq \infty,\;j\in\Z.
     $$
  \end{prop}
    \begin{proof}
    In order to use Lemma~\ref{lem:Tdelta} for $\xi_0\in S$ we set $\eta_i(z):=\eta(z) |z|^{-s_i}$ for
    $z\in\R$.
    Then $\eta\in C_0^\infty(\R),0\notin \supp(\eta)$ implies $\eta_i\in
    C_0^\infty(\R)$ for $i=1,2$. Moreover, we have for $i=1,2$ and $j\in\Z$
  \begin{align*}
    \mathcal T_j u
    &= \mathcal F^{-1}\left( \eta(2^j|\xi-\xi_0|) \hat u_2(\xi)\right) \\
    &=   \mathcal F^{-1}\left( \eta_i(2^j|\xi-\xi_0|)\, (2^j|\xi-\xi_0|)^{s_i}\,\hat
    u_2(\xi)\right)  \\
    &= 2^{js_i} \mathcal F^{-1}\left(\eta_i(2^j|\xi-\xi_0|) m_i(\xi) P_i(\xi)\hat u(\xi)\right)
  \end{align*}
  where $m_i(\xi) := (1-\tau(\xi))|\xi-\xi_0|^{s_i}/P_i(\xi)$. 
  Since $\tau$ is smooth and identically 1 near $\xi_0\in S$, a calculation shows that our assumptions on
  $P_i$ from (A2) imply that $m_i$ satisfies the assumptions of Proposition~\ref{prop:multiplier}. In fact, for
  $|\alpha|=k:=\lfloor d/2\rfloor +1$ and $Q_i,\eps>0$ as in assumption (A2),  
  \begin{align*}
    |\partial^\alpha m_i(\xi)|
    &\les \sum_{0\leq \gamma\leq \alpha} \vecII{\alpha}{\gamma} \left|\partial^{\alpha-\gamma} 
    \left((1-\tau(\xi))|\xi-\xi_0|^{s_i} \langle \xi\rangle^{-s_i}\right)\right| |\partial^\gamma Q_i(\xi)| \\
    &\les  1\cdot |\partial^\alpha Q_i(\xi)|  
     +  \sum_{0\leq \gamma< \alpha} \langle \xi\rangle^{-|\alpha-\gamma|-1} 
     |\partial^\gamma Q_i(\xi)| \\
    &\les \langle \xi\rangle^{-\eps-|\gamma|}    
     +  \langle \xi\rangle^{-|\alpha-\gamma|-1} \langle \xi\rangle^{-|\gamma|} \\
    &\les \langle \xi\rangle^{-\min\{1,\eps\}-|\alpha|}.    
  \end{align*} 
  Here we used the Leibniz rule. So, by Proposition~\ref{prop:multiplier}, $m_i$ is an $L^\mu$-multiplier
  for all $\mu\in [1,\infty]$.
  Hence, Lemma~\ref{lem:Tdelta} yields for all $q\in [p,\infty]$ 
  \begin{align*}
    \|\mathcal T_ju\|_{q}
    &\les 2^{j s_i} \|\mathcal F^{-1}(\eta_i(2^j|\xi-\xi_0|) m_i(\xi)
    \widehat{P_i(D)u}(\xi))\|_{q}  \\
    &\les 2^{j(s_i-d(\frac{1}{p}-\frac{1}{q}))} \|\mathcal F^{-1}(  
    m_i(\xi) \widehat{P_i(D)u}(\xi))\|_p \\
    &\les 2^{j(s_i-d(\frac{1}{p}-\frac{1}{q}))} \|P_i(D)u\|_p.
  \end{align*}
  \end{proof}
  
  Next we use these dyadic estimates to prove estimates of Gagliardo-Nirenberg type. We deduce our
  results from a detailed analysis of the special case $P_i(D)=\bra{D}^{s_i}$ for $s_1,s_2\in\R$. This is
  possible due to
  \begin{equation} \label{eq:BesselPrototype}
    \|\bra{D}^{s_i}u_2\|_p \les \|P_i(D)u\|_p \qquad (1\leq p\leq \infty)
  \end{equation}
  for symbols $P_1,P_2$ as in (A2) thanks to Proposition~\ref{prop:multiplier}.  So we collect some mapping
  properties of the Bessel potential operators $\bra{D}^{-s}$ where $s>0$.
  
  \begin{prop} \label{prop:BesselPotentials}
    Assume $d\in\N, s>0$ and $p,q,r\in [1,\infty], u\in\mathcal S(\R^d)$. 
    \begin{itemize} 
      \item[(i)] If $0\leq \frac{1}{p}-\frac{1}{q}<\frac{s}{d}$ then  $\|u\|_q \les \|\bra{D}^su\|_p$.
      \item[(ii)] If $0\leq \frac{1}{p}-\frac{1}{q}=\frac{s}{d}$ 
      and $1<p,q<\infty$ then $\|u\|_{q,r} \les \|\bra{D}^s u\|_{p,r}$ and $\|u\|_q\les \|\bra{D}^s u\|_p$.
      \item[(iii)] If $0\leq \frac{1}{p}-\frac{1}{q}=\frac{s}{d}$ and $s=d=1$ 
      then $\|u\|_\infty \les \|\bra{D} u\|_1$.      
      \item[(iv)] If  $0\leq \frac{1}{p}-\frac{1}{q}=\frac{s}{d}$ and $1=p<q<\infty$ then  
      $\|u\|_{q,\infty} \les \|\bra{D}^s u \|_1$. 
    \end{itemize}
  \end{prop}
  \begin{proof}
    The parts (i),(iv) and the second part of (ii) are given in \cite[Corollary~1.2.6]{Graf_Modern}; the
    Lorentz space mapping properties from (ii) follow from real interpolation. The estimate (iii) follows
    from 
    $$ 
      \|u\|_\infty
      \les \|u'\|_1 
      = \|m(D)(\bra{D} u)\|_1
      \les \|\bra{D}u\|_1\qquad (u\in\mathcal S(\R)).
    $$
    Here we used that $m(\xi):= \xi(1+|\xi|^2)^{-1/2}$ satisfies the assumptions of
    Proposition~\ref{prop:multiplier}.   
  \end{proof}
%   \begin{proof}
%     (v)    \cite[Theorem~3.5.4]{BerghLoefstrom1976} gives       
%       $L^q(\R^d) = (L^{q,\infty}(\R^d),L^{q,\kappa q}(\R^d))_{\kappa,q}$ for $0<\kappa<1$.
%       On the other hand (iii) gives $H^{s_1}_1\subset L^{q,\infty}(\R^d)$ and (ii) gives
%       $H^{s_2}_{r_2}\subset L^{q,r_2}(\R^d)\subset L^{q,\kappa q}(\R^d)$ if $\frac{r_2}{q}\leq \kappa\leq 1$.
%       Combining these facts we get the claim. \\
%     (vi) \m{$(r_1,r_2)\neq (\infty,\infty)$} Stein Interpolation \m{Ref} to  $T(z) :=
%     \bra{D}^{(1-z)s_1+zs_2}$.
%     We have $$
%      |T(it)|\les \|\bra{D}^{s_1}u\|_{r_1},\qquad
%      |T(1+it)|\les \|\bra{D}^{s_2}u\|_{r_2}.  
%    $$  
%      Hence, for $\frac{1}{\ov r}=\frac{1-\kappa}{r_1}+\frac{\kappa}{r_2}$, 
%      $$
%        \|\bra{D}^{\ov s}u \|_{\ov r} 
%        \les \|\bra{D}^{s_1}u\|_{r_1}^{1-\kappa}\|\bra{D}^{s_2}u\|_{r_2}^{\kappa}.
%      $$
%      Our assumptions imply $\ov r<\infty$. So if $1<q<\infty$  then Sobolev's Embedding Theorem implies
%      $$
%        \|u\|_{q,\ov r}
%        \les \|\bra{D}^{\ov s}u \|_{\ov r} 
%        \les \|\bra{D}^{s_1}u\|_{r_1}^{1-\kappa}\|\bra{D}^{s_2}u\|_{r_2}^{\kappa}.
%      $$
%   \end{proof} 
    
 We finally use these estimates to prove Gagliardo-Nirenberg inequalities for large frequencies. 
%  We shall
%  see in Appendix~A that these estimates, with the exception of the special case $d=1$,
%  $\{(r_1,s_1),(r_2,s_2)\}=\{(1,1),(\infty,0)\}$ in part (i), come from ``optimal'' estimates for the
%  prototypical examples $P_i(D)=\bra{D}^{s_i}$.  
 
\begin{prop}  \label{prop:LargeFreqInterpolated}
  Assume $d\in\N$, $\kappa\in [0,1]$ and (A2) for $s_1,s_2\in\R$. Then 
  \begin{equation} \label{eq:GNLargeFreq}
     \|u_2\|_q  \les \|P_1(D)u \|_{r_1}^{1-\kappa} \|P_2(D)u\|_{r_2}^\kappa
     \qquad (u\in\mathcal S(\R^d))
  \end{equation}
  holds provided that the exponents $q,r_1,r_2\in [1,\infty]$ satisfy 
  $0\leq \frac{1-\kappa}{r_1}+\frac{\kappa}{r_2}-\frac{1}{q}   \leq   \frac{\ov s}{d}$ as well as the
  following conditions in the endpoint case $\frac{1-\kappa}{r_1}+\frac{\kappa}{r_2}-\frac{1}{q}  = 
  \frac{\ov s}{d}$:
  \begin{itemize}
    \item[(i)] if $q=\infty$ then $\frac{1}{r_1}- \frac{s_1}{d}\neq 0 \neq \frac{1}{r_2}-\frac{s_2}{d}$ or
    $r_1=r_2=\infty,s_1=s_2=0$ or $d=1, (r_1,r_2)=(\frac{1}{s_1},\frac{1}{s_2}), s_1,s_2\in\{0,1\}$,
    \item[(ii)] if $1<q<\infty$ and $\frac{1}{r_1}-\frac{s_1}{d}=\frac{1}{q}=\frac{1}{r_2}-\frac{s_2}{d}$ and
    if $r_1=1,\kappa<1$   then \\ $1<r_2<q,\, \kappa\geq \frac{r_2}{q}$ or 
    $r_2=\infty,\frac{1}{q}\leq \kappa\leq \frac{1}{q'}$,
    \item[(iii)] if $1<q<\infty$ and $\frac{1}{r_1}-\frac{s_1}{d}=\frac{1}{q}=\frac{1}{r_2}-\frac{s_2}{d}$
    if $r_2=1,\kappa>0$   then \\ $1<r_1<q,\,1-\kappa\geq \frac{r_1}{q}$ or $r_1=\infty, 
     \frac{1}{q}\leq 1-\kappa\leq \frac{1}{q'}$.
  \end{itemize}
\end{prop}
\begin{proof}
  As mentioned before, it is sufficient to prove the estimates in the prototpyical case
  $P_i(D)=\bra{D}^{s_i}$. So the case $\kappa\in\{0,1\}$ is covered by
  Proposition~\ref{prop:BesselPotentials}~(i),(ii),(iii). So we may concentrate on $\kappa\in (0,1)$ in the
  following. We combine Proposition~\ref{prop:dyadicEstimatesI} and Lemma~\ref{lem:SummationLemma} for
  the Bessel potential spaces $X_i:= P_i(D)^{-1}L^{r_i}(\R^d) = \bra{D}^{-s_i}L^{r_i}(\R^d)$ and $i=1,2$. Here
  we use the identity
   \begin{align*}
    u_2 = \sum_{j=-\infty}^{j_0} \mathcal T_j u
    \quad\text{where}\quad
    \|\mathcal T_ju\|_{q_i}\les 2^{j(s_i-d(\frac{1}{r_i}-\frac{1}{q_i}))} \|u\|_{X_i} \qquad (j\in\Z,
    1\leq r_i\leq q_i\leq \infty),
  \end{align*}
  see Proposition~\ref{prop:dyadicEstimatesI}.
  Our strategy is as follows. We first prove apply Lemma~\ref{lem:SummationLemma} to get strong bounds. This
  will cover all non-endoint cases $0\leq \frac{1-\kappa}{r_1}+\frac{\kappa}{r_2}-\frac{1}{q}<\frac{\ov
  s}{d}$ as well as the endpoint cases involving $q\in\{1,\infty\}$.
  The remaining discussion for $1<q<\infty$ and $1<r_1,r_2<\infty$ can be taken from the literature, but
  the analysis for $\{r_1,r_2\}\cap \{r_1,r_2\}\neq \emptyset$ is more delicate. 
  We will first address the case $\frac{1}{r_1}-\frac{1}{r_2}=\frac{s_1-s_2}{d}$ where we prove our claim using complex and real
  interpolation theory. Finally, in the case $\frac{1}{r_1}-\frac{1}{r_2}\neq \frac{s_1-s_2}{d}$ we will first deduce
  restricted weak-type bounds from Lemma~\ref{lem:SummationLemma} and upgrade them to strong bounds by
  interpolating the restricted weak-type bounds with each other. We will need in the
  following that our assumptions imply $\ov s\geq 0$.

  \medskip

  \textbf{Step 1:}\;  We start the interpolation procedure with (non-endpoint) exponents satisfying
  \begin{align}\label{eq:LargeFreqI}
    0\leq \frac{1-\kappa}{r_1}+\frac{\kappa}{r_2}-\frac{1}{q}<  \frac{\ov s}{d}.
  \end{align}
  In that case the interpolation estimate~\eqref{eq:SummationII} with 
  $(Y_1,Y_2,\theta,r):=(L^{q_1}(\R^d),L^{q_2}(\R^d),\kappa,q)$
  gives the bound
  $$
    \| u_2\|_q
    = \|\sum_{j=-\infty}^{j_0} \mathcal T_j u\|_q
     \stackrel{\eqref{eq:SummationII}}\les \|u\|_{(X_1,X_2)_{\kappa,q}}
      \stackrel{\eqref{eq:interpolationfunctor}}\les
      %\|u\|_{X_1}^{1-\kappa} \|u\|_{X_2}^\kappa \\  =
     \| \bra{D}^{s_1}u\|_{r_1}^{1-\kappa} \| \bra{D}^{s_2} u\|_{r_2}^\kappa.
  $$
  Here, \eqref{eq:SummationII} applies because  \eqref{eq:LargeFreqI} allows to   find $q_i\in
  [r_i,\infty]$ such that
  \begin{align*}
    (1-\kappa)\left( s_1-d\left(\frac{1}{r_1}-\frac{1}{q_1}\right)\right)
    + \kappa \left( s_2-d\left(\frac{1}{r_2}-\frac{1}{q_2}\right)\right)
    > 0,\quad
    \frac{1}{q}=\frac{1-\kappa}{q_1}+\frac{\kappa}{q_2}.
  \end{align*}
  So the claim is  proved for all non-endpoint exponents given by~\eqref{eq:LargeFreqI}.
  \medskip

  It remains to discuss the endpoint case $0\leq \frac{1-\kappa}{r_1}+\frac{\kappa}{r_2}-\frac{1}{q} =
  \frac{\ov s}{d}$. Using~\eqref{eq:SummationI} for $Y_1=Y_2=L^q(\R^d)$   we get the claim for all exponents
  satisfying
  \begin{align}\label{eq:LargeFreqII}
    &0\leq \frac{1-\kappa}{r_1}+\frac{\kappa}{r_2}-\frac{1}{q} = \frac{\ov s}{d}
    \quad\text{and}\quad q\geq \max\{r_1,r_2\},\; \frac{1}{r_1}-\frac{s_1}{d}\neq \frac{1}{q} \neq
    \frac{1}{r_2}-\frac{s_2}{d}.
  \end{align}
  Here the latter two inequalities correspond to $\beta_1,\beta_2\neq 0$ in
  Lemma~\ref{lem:SummationLemma}.  From this we infer that  the claimed endpoint estimates hold for $q\in\{1,\infty\}$  via the following case
  distinction:
  \begin{itemize}
    \item Case $q=1$: \quad   $r_1=r_2=1,s_1=s_2=0$ is trivial,
    \item Case $q=1$: \quad   $r_1=r_2=1,\ov s=0,s_1\neq 0\neq s_2$ is covered by~\eqref{eq:LargeFreqII},
    \item Case $q=\infty$: \quad $r_1=r_2=\infty, s_1=s_2=0$ is trivial,
    \item Case $q=\infty$: \quad $\frac{1}{r_1}-\frac{s_1}{d}\neq 0\neq
    \frac{1}{r_2}-\frac{s_2}{d}$ is covered by~\eqref{eq:LargeFreqII}, 
    \item Case $q=\infty$: \quad $(d,r_1,r_2)=(1,\frac{1}{s_1},\frac{1}{s_2}), s_1,s_2\in\{0,1\}$ 
    is covered by Proposition~\ref{prop:BesselPotentials}~(iii).
  \end{itemize}
  These are all cases involving $q\in\{1,\infty\}$ and in particular claim (i) is proved. 
  So we are left with those endpoint estimates for $1<q<\infty$ that
  are not covered by~\eqref{eq:LargeFreqII}. 
  
  \medskip
  
   \textbf{Step 2:}\;  The claim holds for $1<r_1,r_2<\infty$ due to
  $$
    \|u\|_q
    \les \|\bra{D}^{\ov s}u\|_{\ov r}
    \les \|\bra{D}^{s_1}u\|_{r_1}^{1-\kappa} \|\bra{D}^{s_2}u\|_{r_2}^{\kappa},
  $$
  where $\frac{1}{\ov r}:=\frac{1-\kappa}{r_1}+\frac{\kappa}{r_2}$.
  This is a consequence of Sobolev's Embedding Theorem \cite[Theorem~6.5.1]{BerghLoefstrom1976}
  and the complex interpolation result from~\cite[Theorem~6.4.5~(7)]{BerghLoefstrom1976}.
  So we may in the following assume $\{r_1,r_2\}\cap \{1,\infty\}\neq \emptyset$. As announced earlier, we
  first deal with $\frac{1}{r_1}-\frac{1}{r_2}=\frac{s_1-s_2}{d}$.

  \medskip
  
   \textbf{Step 3:}\;    So assume we are in the endpoint case with $1<q<\infty , \frac{1}{r_1}-\frac{1}{r_2}=\frac{s_1-s_2}{d}$,
  $r_1\leq r_2$ (w.l.o.g.) and $\{r_1,r_2\}\cap \{1,\infty\}\neq \emptyset$.  Then
  $\frac{1-\kappa}{r_1}+\frac{\kappa}{r_2}-\frac{1}{q}=\frac{\ov s}{d}$ implies $\frac{1}{r_1}-\frac{s_1}{d}= \frac{1}{q}=\frac{1}{r_2}-\frac{s_2}{d}$. We distinguish the following cases: 
  \begin{itemize}
    \item Case $r_1=1,r_2=1$:\; This case is excluded, so there is nothing to prove.
    \item Case $r_1=1,1<r_2<q$:\; By Proposition~\ref{prop:BesselPotentials}~(ii),(iv) 
    we have $\|u\|_{q,\infty} \les \|\bra{D}^{s_1}u\|_1$ as well as $\|u\|_{q,r_2}\les
    \|\bra{D}^{s_2}u\|_{r_2}$. Applying the  interpolation identity \cite[Theorem~5.3.1]{BerghLoefstrom1976}
    \begin{equation}\label{eq:InterpolationIdentity}
      L^q(\R^d)= \left(L^{q,\infty}(\R^d), L^{q,\kappa q}(\R^d)\right)_{\kappa,q}, \qquad \kappa\in (0,1],
    \end{equation}    
     we infer for all 
    $\kappa\in [\frac{r_2}{q},1]$ 
    \begin{align*}
      \|u\|_q 
      \les \|u\|_{q,\infty}^{1-\kappa}  \|u\|_{q,\kappa q}^\kappa 
      \les \|u\|_{q,\infty}^{1-\kappa}  \|u\|_{q,r_2}^\kappa 
      \les \|\bra{D}^{s_1}u\|_1^{1-\kappa} \|\bra{D}^{s_2}u\|_{r_2}^\kappa.
    \end{align*} 
    \item Case $r_1=1,r_2=\infty$:\; We have to prove \eqref{eq:GNLargeFreq} 
    for $\frac{1}{q}\leq \kappa\leq
    \frac{1}{q'}$. It is sufficient to prove the claim first for $\kappa=\frac{1}{q}$ and then for
    $\kappa=\frac{1}{q'}$. We   use $\|u\|_{q,\infty} \les
    \|\bra{D}^{s_1}u\|_1$ and
     \begin{equation} \label{eq:interpol0}
      \|u\|_{q,2}^2
      \les \|\bra{D}^{\frac{d}{2}-\frac{d}{q}}u\|_2^2 
      = \int_{\R^d} \bra{D}^{\frac{d}{q'}}u\cdot \bra{D}^{-\frac{d}{q}}u\,dx
      \leq \| \bra{D}^{s_1} u\|_1  \|\bra{D}^{s_2} u\|_\infty.
    \end{equation}
      In \eqref{eq:interpol0} we subsequently used Propostiion~\ref{prop:BesselPotentials}~(ii), the
      $L^2$-isometry property of the Fourier transform as well as $s_1=\frac{d}{q'},s_2=-\frac{d}{q}$. 
    Real interpolation of these two estimates and  $L^q(\R^d)= (L^{q,\infty}(\R^d),
    L^{q,2})_{2/q,q}$, which is \eqref{eq:InterpolationIdentity} for $\kappa=\frac{2}{q}$,  gives
    \begin{equation} \label{eq:interpol1}
      \|u\|_q 
      \les \|u\|_{q,\infty}^{1-\frac{2}{q}}  \|u\|_{q,2}^{\frac{2}{q}}
      \les \|\bra{D}^{s_1}u\|_1^{\frac{1}{q'}}  \|\bra{D}^{s_2}u\|_{\infty}^{\frac{1}{q}}.
    \end{equation} 
    So the claim holds for $\kappa=\frac{1}{q}$ and we now consider $\kappa=\frac{1}{q'}$. Here we use Stein's
    Interpolation Theorem \cite{Stein_Interpolation} in a more general setting 
    \cite[Theorem~2.1]{Voigt_Stein} for the family of linear operators $\mathcal T^s u := e^{s^2} \bra{D}^{s/2
    - d/q}u$ with $s\in\C,0\leq \Real(s)\leq 1$.
    We have
    \begin{align*}
      \|\mathcal T^{it}u\|_{\BMO(\R^d)} 
      &= e^{-t^2} \|\bra{D}^{it}(\bra{D}^{-\frac{d}{q}}u)\|_{\BMO(\R^d)} 
       \les \|\bra{D}^{-\frac{d}{q}}u\|_{\infty}, \\
      \|\mathcal T^{1+it}u\|_{2} 
      &= e^{1-t^2} \|\bra{D}^{\frac{d}{2}-\frac{d}{q}}u\|_2 
       \stackrel{\eqref{eq:interpol0}}\les \| \bra{D}^{\frac{d}{q'}}u\|_1^{\frac{1}{2}} 
       \|\bra{D}^{-\frac{d}{q}}u\|_\infty^{\frac{1}{2}}.
    \end{align*}
    Here we used the validity of Mikhlin's Multiplier Theorem in $\BMO(\R^d)$ to deduce that
    the operator norm $\bra{D}^{it}:L^\infty(\R^d)\to \BMO(\R^d)$ is polynomially bounded 
    with respect to  $t$ and thus
    compensated by the mitigating factor $e^{-t^2}$ as $|t|\to\infty$. We refer to  Proposition~3.4,
    Theorem~4.4 and the comments on page 20-21 in Tao's Lecture notes~\cite{Tao} where such an application in the context of
    Stein's interpolation theorem is explicitly mentioned. In view of $[\BMO(\R^d),L^2(\R^d)]_\theta =
    L^{2/\theta}(\R^d)$ for $0<\theta\leq 1$ we may plug in $\theta=\frac{2}{q}$ and get 
    in view of  $s_1=\frac{d}{q'},s_2=-\frac{d}{q}$ 
    $$
      \|u\|_q
      = \|\mathcal T^{\frac{2}{q}}u\|_{q} 
      \les \|\bra{D}^{-\frac{d}{q}}u\|_{\infty}^{1-\theta}
       \left( \| \bra{D}^{\frac{d}{q'}}u\|_1^{\frac{1}{2}} 
       \|\bra{D}^{-\frac{d}{q}}u\|_\infty^{\frac{1}{2}}\right)^\theta
      =   \|\bra{D}^{s_1}u\|_1^{\frac{1}{q}}  \|\bra{D}^{s_2}u\|_{\infty}^{\frac{1}{q'}}.
    $$
    \item Case $1<r_1<r_2=\infty$:\; We have to prove \eqref{eq:GNLargeFreq} for $1<q<r_1,\kappa\geq
    \frac{r_1}{q}$. We consider $\mathcal T^s u := e^{s^2}  \bra{D}^{s_2+s(s_1-s_2)}u$ and obtain as before
    \begin{align*}
      \|\mathcal T^{it}u\|_{\BMO(\R^d)} 
      \les \|\bra{D}^{s_2} u\|_{\infty},\qquad
      \|\mathcal T^{1+it}u\|_{r_1} \les \|\bra{D}^{s_1}u\|_{r_1}.
    \end{align*}
    So we conclude for $\kappa:= \frac{r_1}{q} = \frac{s_2}{s_2-s_1}$
    $$
      \|u\|_q
      = \|\mathcal T^{\kappa} u\|_{\frac{r_1}{\kappa}}
      \les \|\bra{D}^{s_2} u\|_{\infty}^{1-\kappa} \|\bra{D}^{s_1}u\|_{r_1}^\kappa.
    $$  
    This proves the claim for $\kappa=\frac{r_1}{q}$. Since the desired bound for $\kappa=1$
    follows from Proposition~\ref{prop:BesselPotentials}~(ii), we
    get the claim for $\kappa\in [\frac{r_1}{q},1]$.  
    \item Case $1<r_1=r_2=\infty$: This case does not occur because
    $\frac{1-\kappa}{r_1}+\frac{\kappa}{r_2}-\frac{1}{q}=-\frac{1}{q}<0$.
  \end{itemize}

  \medskip
  
   \textbf{Step 4:}\;  
  To prove the remaining estimates we first prove restricted weak-type estimates $\| u_2\|_{q,\infty} \les
  \|u\|_{(X_1,X_2)_{\kappa,1}}$ for all exponents satisfying
  \begin{align}\label{eq:LargeFreqIII}
    0\leq \frac{1-\kappa}{r_1}+\frac{\kappa}{r_2}-\frac{1}{q} = \frac{\ov s}{d}
    \quad\text{and}\quad 1<q<\infty 
    \quad\text{and}\quad \frac{1}{r_1}-\frac{1}{r_2}\neq \frac{s_1-s_2}{d}.
  \end{align}
  For $s_1=s_2=0$ this is implied by H\"older's Inequality, so we
  may assume $\ov s>0$ or $\ov s=0, (s_1,s_2)\neq (0,0)$. For $\ov s=0,(s_1,s_2)\neq (0,0),q=r_1=r_2$ this is
  implied by the strong estimates in the case~\eqref{eq:LargeFreqII}, so we may even  assume
  $\ov s>0$ or $\ov s=0, (s_1,s_2)\neq (0,0),(r_1,r_2)\neq (q,q)$.
  For the remaining exponents the weak estimate is a consequence of~\eqref{eq:SummationII} because
  one can find $q_i\in [r_i,\infty]$ such that
  \begin{align*}
    &(1-\kappa)\left( s_1-d\left(\frac{1}{r_1}-\frac{1}{q_1}\right)\right)
    + \kappa \left( s_2-d\left(\frac{1}{r_2}-\frac{1}{q_2}\right)\right)
    = 0,\\
    &\frac{1}{q}=\frac{1-\kappa}{q_1}+\frac{\kappa}{q_2},\qquad
     s_i-d\left(\frac{1}{r_i}-\frac{1}{q_i}\right)\neq 0,\quad q_1\neq q_2.
  \end{align*}
  Indeed, this condition is equivalent to $\frac{1-\kappa}{r_1}+\frac{\kappa}{r_2}-\frac{1}{q} = \frac{\ov
  s}{d}$ and finding $q_2$ such that
  \begin{align*}
     \frac{1}{q}-\frac{1-\kappa}{r_1}\leq \frac{\kappa}{q_2}\leq \frac{\kappa}{r_2},\quad
     q_2\neq q,\quad
     \frac{1}{q}-(1-\kappa)\left(\frac{1}{r_1}-\frac{s_1}{d}\right) \neq \frac{\kappa}{q_2}
     \neq \kappa \left(\frac{1}{r_2}-\frac{s_2}{d}\right),
  \end{align*}
  and such a choice is possible due to our assumptions. (In the case $\ov s=0, (s_1,s_2)\neq
  (0,0),(r_1,r_2)\neq (q,q)$ choose $q_2=r_2,q_1=r_1$.) In this way we obtain $\| u_2\|_{q,\infty} \les
  \|u\|_{(X_1,X_2)_{\kappa,1}}$ for all exponents satisfying~\eqref{eq:LargeFreqIII}.  We finally 
  interpolate these  restricted weak-type estimates  with each other to prove strong estimates for exponents
  as in~\eqref{eq:LargeFreqIII}.
  To this end let $\delta>0$ be sufficiently small (but fixed) and $\eps:=
  \delta(\frac{s_1-s_2}{d}-\frac{1}{r_1}+\frac{1}{r_2})\neq 0$ and define $\tilde q,q^*,\tilde\kappa,\kappa^*$ via  $\frac{1}{\tilde
    q}-\eps=\frac{1}{q}=\frac{1}{q^*}+\eps$ and $\tilde \kappa-\delta=\kappa=\kappa^*+\delta$.
    Then $(\tilde q,r_1,r_2,\tilde\kappa),(q^*,r_1,r_2,\kappa^*)$ satisfies
    \eqref{eq:LargeFreqIII} and the reiteration property of real interpolation
    \cite[Theorem~3.5.3]{BerghLoefstrom1976} gives
    \begin{align*}
      \|u_1\|_q
      &\les \|u_1\|_{(L^{q^*}(\R^d),L^{\tilde q}(\R^d))_{\frac{1}{2},q}} \\
      &\les \|u\|_{((X_1,X_2)_{\kappa^*,1},(X_1,X_2)_{\tilde\kappa,1})_{\frac{1}{2},q}} \\
      &\les \|u\|_{(X_1,X_2)_{\kappa,q}} \\
      &\stackrel{\eqref{eq:interpolationfunctor}}\les \|P_1(D)u\|_{r_1}^{1-\kappa} \|P_2(D)u\|_{r_2}^{\kappa}
    \end{align*}
   Here the first bound uses $\frac{1}{q}=\frac{1}{2}(\frac{1}{q^*}+\frac{1}{\tilde q})$ and the third uses
   $\kappa=\frac{1}{2}(\tilde\kappa+\kappa^*)$. This finishes the proof.
\end{proof}

We have thus proved that the  Gagliardo-Nirenberg inequality~\eqref{eq:GNgeneral} holds for non-critical
frequencies whenever the exponents belong to the set 
\begin{align*}
%\end{align*}\label{eq:LargeFreqCondition}
 % \begin{aligned}
    \mathcal B(\kappa) := \left\{  (q,r_1,r_2)\in [1,\infty]^3:\;   (q,r_1,r_2) \text{ as in
    Proposition~\ref{prop:LargeFreqInterpolated}}    
    \right\}.
%\end{aligned}
\end{align*}

\begin{rem} \label{rem:GNandFracLap} ~
  \begin{itemize}
  \item[(a)]
    The original Gagliardo-Nirenberg inequality~
  $
    \|\nabla^j v\|_q \les \|\nabla^m v\|_{r_1}^{1-\kappa} \|v\|_{r_2}^\kappa
  $
  from~\cite[p.125]{Nirenberg_OnElliptic} holds for $j,m\in\N$ provided that $\frac{1}{q}- \frac{j}{d} =
  (1-\kappa)(\frac{1}{r_1}-\frac{m}{d}) + \frac{\kappa}{r_2}$ and $\frac{j}{m}\leq 1-\kappa<1$. Our
  result shows that ``in most cases'' the large frequency part of this estimate holds provided
  that $\frac{j}{m}\leq 1-\kappa<1$ holds and $\frac{1}{q}- \frac{j}{d} \geq
  (1-\kappa)(\frac{1}{r_1}-\frac{m}{d}) + \frac{\kappa}{r_2}$. The exceptions are due to  the fact that,
  in $L^1(\R^d)$ or $L^\infty(\R^d)$, the term $\bra{D}^j u$ does not control $D^j u$, i.e., not every single
  partial derivative of order $j$.  This is a consequence of the unboundedness
  of the Riesz transform on these spaces.
    \item[(b)]
    Our proof indicates which function spaces to choose in order to get some endpoint estimates in the
    exceptional cases as well. Roughly speaking, one may replace $L^q(\R^d)$ by $L^{q,r}(\R^d)$ for
    suitable $r>q$ and $L^\infty(\R^d)$ by $\BMO(\R^d)$ on the left hand side. On the right hand side the
    Hardy space $\mathcal H^1(\R^d)$ may replace $L^1(\R^d)$.
    % A very quick proof of~\eqref{eq:GNLargeFreq}
    %then results from applying Stein's Interpolation Theorem to the analytic family of linear operators
    %$\bra{D}^{s_1+s(s_2-s_1)}$ and using Sobolev's Embedding Theorem.
    \item[(c)] One may as well consider symbols $P_i(D)$ that   vanish at some
    finite set of points in $\R^d\sm S$. If for instance one has 
     $P_i(\xi)=b_i(\xi)|\xi-\xi^*|^{t_i}$ near $\xi^*\in\R^d\sm S$ for $t_1,t_2>-d$ and
     non-vanishing $b_i\in C^\infty(\R^d)$, then one finds as in
     Proposition~\ref{prop:LargeFreqInterpolated} that the interpolation estimate holds in this frequency
     regime whenever $\frac{1-\kappa}{r_1}+\frac{\kappa}{r_2}-\frac{1}{q}>\frac{\ov t}{d}$ 
     where $\ov t:=(1-\kappa)t_1+\kappa t_2$. Under suitable extra conditions similar to the ones above, this
     may be extended to the endpoint case $\frac{1-\kappa}{r_1}+\frac{\kappa}{r_2}-\frac{1}{q}=\frac{\ov
     t}{d}$.
%     \item[(d)] We comment on the optimality of our sufficient conditions from
%     Proposition~\ref{prop:LargeFreqInterpolated}. The presence of the cut-off function $\tau$ makes it
%     impossible to produce explicit counterexamples. However, we verified that the above sufficient conditions, 
%     without the special case $d=1,q=\infty$ and $\{(r_1,s_1),(r_2,s_2)\}=\{(1,1),(\infty,0)\}$, are
%     necessary and sufficient for~\eqref{eq:GNLargeFreq} to hold whenever the right hand side is
%     finite, i.e., without taking care if the additional hypotheses $u\in\mathcal S(\R^d)$. (So the latter
%     assumption is mandatory for $d=1,q=\infty$ and $\{(r_1,s_1),(r_2,s_2)\}=\{(1,1),(\infty,0)\}$.) The
%     analysis of corresponding counterexamples is straightforward but lengthy.  \m{We sketch them in the ArXiv version of
%     this paper}.
    \item[(d)] The proof in the important special case $1<r_1,r_2,q<\infty$ is much shorter than the
    complete analysis, see the beginning of Step 2. 
  \end{itemize} 
\end{rem}

\section{Critical frequency analysis}
  
  We introduce a real number $A_\eps(p,q)$ such that $\|\tilde T_j\|_{p\to q}\les 2^{-jA_\eps(p,q)}$ 
  holds for suitably defined dyadic operators $\tilde T_j$ that play the role of the $T_j$ in the previous
  section.
  Unfortunately, the definition of $A_\eps(p,q)$ is rather complicated for $d\geq 2$. It involves the number 
  $$
    A(p,q):=\min\{A_0,A_1,A_2,A_2',A_3,A_3',A_4,A_4'\}
  $$ 
  where $A_i=A_i(p,q)$ and  $A_i'=A_i(q',p')$ are given by
%   \begin{itemize}
%     \item[] $A_0=1$,
%     \item[] $A_1 = \frac{k+2}{2}\left(\frac{1}{p}-\frac{1}{q}\right)$,
%     \item[] $A_2= -\frac{k}{2} + \frac{k+1}{p}$,
%     \item[] $A_3 =\frac{2d-k}{2}-\frac{2d-k-1}{q}$,
%     \item[] $A_4=  \frac{k+2}{2}\left(\frac{1}{p}-\frac{1}{q}\right)+\frac{2d-k-2}{2}-\frac{2d-k-2}{q}$.
%   \end{itemize}
  $$
     A_0=1, \qquad
    A_1 = \frac{k+2}{2}\left(\frac{1}{p}-\frac{1}{q}\right), \qquad
    A_2 = \frac{k+2}{2} - \frac{k+1}{q} 
  $$
  as well as  
  $$ 
    A_3  =\frac{2d-k}{2}-\frac{2d-k-1}{q},  \qquad
    A_4 =  \frac{k+2}{2}\left(\frac{1}{p}-\frac{1}{q}\right)+\frac{2d-k-2}{2}-\frac{2d-k-2}{q}.
  $$
  The values $A_0,A_1,A_1',A_2,A_2'$ will be important for  $1\leq p\leq 2\leq q\leq \infty$ whereas
  all other exponents satisfying $1\leq p\leq q\leq \infty$ come with $A_3,A_3',A_4,A_4'$. 
  Then we define  for $\eps>0$
  \begin{align}\label{eq:def_Aeps}
     A_\eps(p,q) &:=  
       \frac{1}{p}-\frac{1}{q} \quad\text{if }d=1, \quad\qquad
     A_\eps(p,q):= 
       A(p,q) - \eps \cdot \ind_{(p,q)\in\mathcal E} \quad\text{if }d\geq 2.
  \end{align}
  Here, $\mathcal E$ denotes a set of exceptional points where we do not have strong bounds, but only weak
  bounds or restricted weak-type bounds. It is given by
  \begin{align*}
     \mathcal E &:= \Big\{(p,q)\in [1,\infty]^2 :\;
    \frac{1}{p}=\frac{k+2}{2(k+1)},\; \frac{1}{q}\leq 
    \frac{k^2}{2(k+1)(k+2)} \quad\text{or } \\
    &\hspace{4cm} \frac{1}{q}=\frac{k}{2(k+1)},\; \frac{1}{p}\geq
    \frac{k^2+6k+4}{2(k+1)(k+2)}\Big\}
  \end{align*}
  and coincides with the red points in Figure~\ref{fig:Tdeltabounds}.  
   
  \medskip 
  
\begin{figure}[htbp] 
%% k=2,N=3
\begin{tikzpicture}[scale=10]
\draw[->] (0,0) -- (1.05,0) node[right]{$\frac{1}{p}$};
\draw[->] (0,0) -- (0,1.05) node[above]{$\frac{1}{q}$};
\draw (0,0) --(1,1);
\draw (1,1) -- (1,0) node[below]{$1$};
\draw (0,1) node[left]{$1$};
\coordinate (O) at (0,0);
\coordinate (O') at (1,1);
\coordinate (ST) at (0.5,0.25); %%$0.25=\frac{k}{2(k+2)}$
\coordinate (ST') at (0.75,0.5);
\coordinate (R) at (0.67,0.17); %% \frac{2}{3} = \frac{k+2}{2(k+1)}, \frac{1}{6} = \frac{k^2}{2(k+1)(k+2)}
\coordinate (R') at (0.83,0.33);
\coordinate (H) at (0.5,0); %%$0.25=\frac{k}{2(k+2)}$
\coordinate (H') at (1,0.5);
\coordinate (C) at (0.5,0.5);
\coordinate (E) at (1,0);
\coordinate (ER) at (0.67,0);
\coordinate (ER') at (1,0.33);
\draw (0.85,0.15) node [rounded corners=8pt]
{$A_0$};
\draw (0.65,0.4) node [rounded corners=8pt]
{$A_1$};
\draw (0.9,0.4) node [rounded corners=8pt]
{$A_2$};
\draw (0.6,0.1) node [rounded corners=8pt]
{$A_2'$};
\draw (0.95,0.75) node [rounded corners=8pt]
{$A_3$};
\draw (0.25,0.05) node [rounded corners=8pt]
{$A_3'$};
\draw (0.7,0.6) node [rounded corners=8pt]
{$A_4$};
\draw (0.4,0.3) node [rounded corners=8pt]
{$A_4'$};
\draw [dotted] (0,0.5) node[left]{$\frac{1}{2}$} -- (0.5,0.5);
\draw [dotted] (0,0.25) node[left]{$\frac{k}{2(k+2)}$} -- (ST);
\draw [dotted] (0,0.17) node[left]{$\frac{k^2}{2(k+1)(k+2)}$} -- (R);
\draw [dotted] (0.75,0) node[below]{\hspace{3mm}$\frac{k+4}{2(k+2)}$} -- (ST');
\draw [dotted] (0.5,0) node[below]{$\frac{1}{2}$} -- (0.5,0.5);
\draw  (C) -- (H);
\draw  (C) -- (H');
\draw  (C) -- (O);
\draw   (C) -- (O');
\draw  (ST) -- (O);
\draw  (ST') -- (O');
\draw  (ST) -- (R);
\draw  (ST') -- (R');
\draw   (R) -- (R');
\draw  (R) -- (ER);
\draw     (R') -- (ER');
\draw [line width = 0.4mm, draw=red] (R) -- (0.67,0) node[below]{\hspace{-3mm}$\frac{k+2}{2(k+1)}$};
\draw [line width = 0.4mm, draw=red] (R') -- (1,0.33);
\end{tikzpicture}
 \caption{Riesz diagram  with the bounds for the mapping constant of $\tilde T_j$ from
 Lemma~\ref{lem:Tdelta2}. The exceptional points from $\mathcal E$ are coloured in red.}
 \label{fig:Tdeltabounds}
\end{figure}
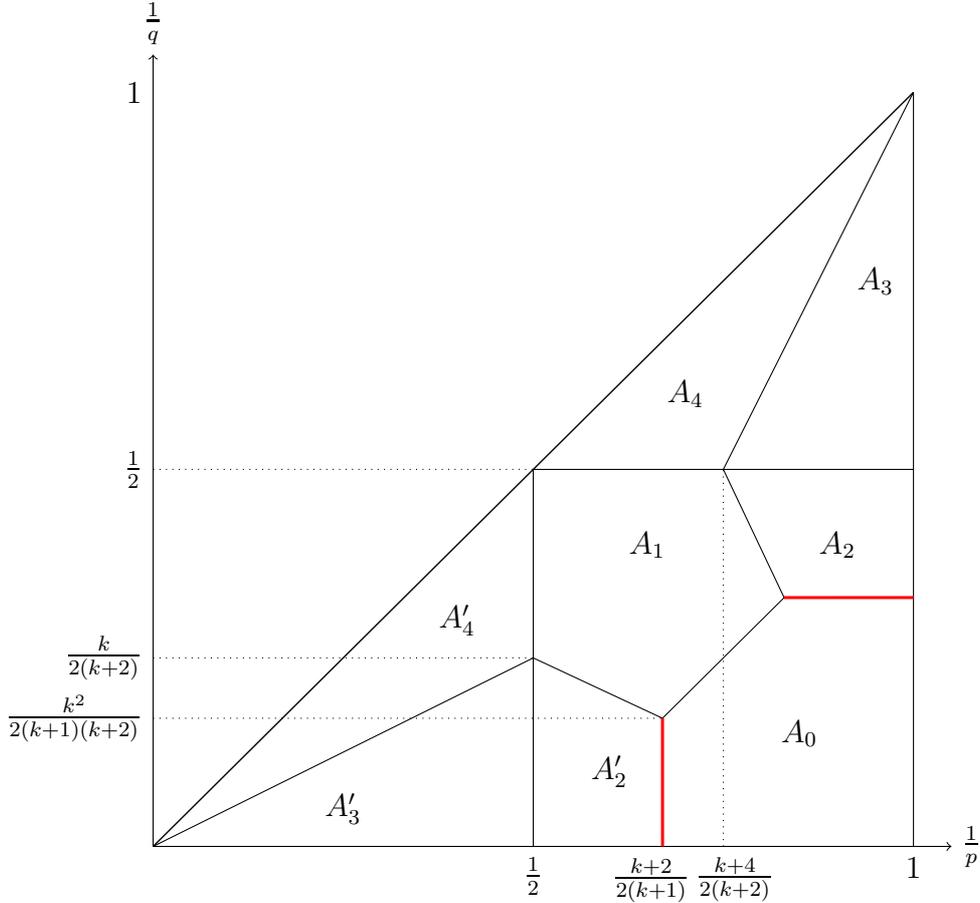
 
  \medskip
  
  We first prove dyadic estimates in the  frequency regime close to the
  critical surface $S$. The latter can be locally parametrized as a graph $\xi_d=\psi(\xi')$ after
  some permutation of coordinates, where $\xi=(\xi',\xi_d)\in\R^{d-1}\times\R\simeq  \R^d$. In view of (A1)
  we  study operators of the form
  \begin{align}\label{eq:def_TildeTj}
    \begin{aligned}
    \tilde T_j f 
    &:= \mathcal F^{-1}\left( \eta\left(2^j(\xi_d-\psi(\xi'))\right) \chi(\xi')\hat f(\xi) \right)
    = \tilde K_j\ast f \\
    \text{where}\quad
    \tilde K_j
    &:= \mathcal F^{-1}\left( \eta \left(2^j(\xi_d-\psi(\xi'))\right) \chi(\xi')\right) 
    \end{aligned}
  \end{align}
   and
   \begin{align} \label{eq:chipsi}
     \begin{aligned}
     &\psi\in C^\infty(\R^{d-1}),\;\chi\in C_0^\infty(\R^{d-1}) \text{ and at least }k\in\{1,\ldots,d-1\} 
     \\
     &\text{eigenvalues of the Hessian } D^2\psi \text{ are non-zero on }\supp(\chi).
     \end{aligned}
   \end{align}
   In the degenerate case $d=1$ we interpret  $\eta (2^j(\xi_d-\psi(\xi'))) \chi(\xi')$ as 
   $\eta(2^j(\xi-c))$ for some constant $c\in\R$.
   Our analysis of the mapping properties of $\tilde T_j$ follows~\cite[Section~4]{ManSch}. 
   Contrary to the situation for $T_j$, only the bounds for $j\nearrow +\infty$ will be of importance.
   Repeating the proof of Lemma~\ref{lem:Tdelta} gives the following result in the one-dimensional case.
  
  \begin{lem}\label{lem:Tdelta2d=1}
      Assume $d=1$ and $\eta\in C_0^\infty(\R)$. Then we have  
  $$
    \|\tilde T_j\|_{p\to q}
     \les 2^{-j(\frac{1}{p}-\frac{1}{q})}  
     \qquad \text{for }1\leq  p\leq q\leq \infty,\;j\in\Z. 
    $$
  \end{lem}
%   \begin{proof}
% 	We have $\|\tilde K_j\|_r=2^{-j}\|\mathcal F^{-1}\eta(2^{-j}\cdot)\|_r\les 2^{-j/r'}$. Choosing $r\in
% 	[1,\infty]$ according to $1+\frac{1}{q}=\frac{1}{r}+\frac{1}{p}$, Young's Convolution
% 	Inequality implies 
% 	$$ 
% 	 \|\tilde T_j\|_{p\to q} \les \|\tilde K_j\|_r\les 2^{-\frac{j}{r'}} = 2^{-j(\frac{1}{p}-\frac{1}{q})}.
%     $$ 
%   \end{proof}
  
  The bounds in  higher dimensions are   more complicated and depend on the number
  $k\in\{1,\ldots,d-1\}$ of non-vanishing principal curvatures of $S$. We first analyze the kernel function
  $\tilde K_j$ following~\cite[Lemma~4.4]{ManSch}.
   
  \begin{prop}\label{prop:kernel}
     Assume $d\in\N,d\geq 2$, let $\chi,\psi,k$ be as in~\eqref{eq:chipsi} and $\eta\in C_0^\infty(\R)$.
     Then the kernel function $\tilde K_j$ satisfies for $j\in\Z, j\geq j_0$
     \begin{equation}\label{eq:normboundsTildeKj}
       \|\tilde K_j\|_r \les 2^{-j(\frac{2d-k}{2}-\frac{2d-k-1}{r})} \;\text{if }1\leq r\leq 2,\qquad
       \|\tilde K_j\|_\infty \les 2^{-j}.
     \end{equation}
  \end{prop}
  \begin{proof}
    The bound $\|\tilde K_j\|_2 \les 2^{-j/2}$ follows from Plancherel's identity and~\eqref{eq:def_TildeTj}.
    Indeed,
    \begin{align*}
      \|\tilde K_j\|_2^2 
      &= \int_{\R^d} \eta(2^j(\xi_d-\psi(\xi')))^2 \chi(\xi')^2\,d(\xi',\xi_d) \\
      &= \int_{\R^{d-1}} \chi(\xi')^2 \left(\int_\R \eta(2^j t)^2\,dt\right) \,d\xi' \\
      &= 2^{-j} \|\chi\|_2^2 \|\eta\|_2^2.
    \end{align*}
    To prove \eqref{eq:normboundsTildeKj} it thus suffices to show $\|\tilde K_j\|_1 \les
    2^{-j(\frac{k+2}{2}-d)}$ as well as $\|\tilde K_j\|_\infty \les 2^{-j}$ and to apply the Riesz-Thorin
    interpolation theorem. These two norm bounds for the kernel function are consequences of the pointwise
    bounds for arbitrary $N,M\in\N_0$
     \begin{align*}%\label{eq:boundsTildeKj}
       \begin{aligned}
       |\tilde K_j(x)| 
       &\les_{N,M} 2^{-j} (1+2^{-j}|x_d|)^{-M}  (1+|x'|)^{-N}
       %&\les_{M,N} 2^{-j} \r{(1+2^{-j}|x_d|)^{-M}}  (1+|x|)^{-N}
       &&\text{if }|x'| \geq c|x_d|, \\
       |\tilde K_j(x)| &\les_M 2^{-j} (1+2^{-j}|x_d|)^{-M}  (1+|x_d|)^{-\frac{k}{2}}
       &&\text{if }|x'| \leq c|x_d|,   
       \end{aligned}
     \end{align*}
     where $c>0$ is suitably chosen. Indeed, choosing $M,N$ sufficiently large we get
  \begin{align*}
   \|\tilde K_j\|_1
   &\les_{N,M} \int_\R
   \left( \int_{|x'|\leq c x_d}  2^{-j} (1+2^{-j}|x_d|)^{-M}  (1+|x_d|)^{-\frac{k}{2}} \,dx'\right)\,dx_d \\
   &\qquad
   + \int_\R \left(\int_{|x'|\geq c x_d}  2^{-j} (1+2^{-j}|x_d|)^{-M}  (1+|x'|)^{-N} \,dx'
   \right)\,dx_d \\
   &\les_{M,N} 2^{-j} \int_\R  (1+2^{-j}|x_d|)^{-M} |x_d|^{d-1}(1+|x_d|)^{-\frac{k}{2}}  \,dx_d \\
   &\qquad + 2^{-j}  \int_\R  (1+2^{-j}|x_d|)^{-M} (1+|x_d|)^{d-N}   \,dx_d \\
   &\les_{M,N} 2^{-j} \int_0^{2^j}
      |x_d|^{d-1}(1+|x_d|)^{-\frac{k}{2}} \,dx_d 
     + 2^{(M-1)j} \int_{2^j}^\infty  |x_d|^{d-\frac{k}{2}-1-M}  \,dx_d  \\
   &\les_{M,N}   2^{-j(\frac{k+2}{2}-d)}.
  \end{align*}
  Here we used $2^j\geq 2^{j_0}>0$. So it remains to prove the pointwise bounds by adapting
  the arguments from \cite{ManSch}. We have 
  $$
      \tilde K_j(x)
      = c_d\, 2^{-j} (\mathcal F^{-1}\eta)(2^{-j}x_d)  \int_{\R^{d-1}}
        e^{i(x'\cdot\xi'+x_d\psi(\xi'))} \chi(\xi')\,d\xi'
    $$
    for some dimensional constant $c_d>0$.
    We choose $c>0$ so large  that the smooth phase function $\Phi(\xi')=x'\cdot\xi' +
    x_d\psi(\xi')$ satisfies $|\nabla \Phi(\xi')|\geq c^{-1}|x'|$ for all $\xi'\in\R^{d-1}$ whenever $|x'|\geq
    c|x_d|$.
    In view of $\chi\in C_0^\infty(\R^{d-1})$  the method of non-stationary phase gives
    \begin{align*}
      |\tilde K_j(x)|
      &\les_N 2^{-j} |(\mathcal F^{-1}\eta)(2^{-j}x_d)| (1+|x'|)^{-N} \\
      &\les_{N,M} 2^{-j} (1+2^{-j}|x_d|)^{-M}  (1+|x'|)^{-N}
      \quad\text{for }|x'|\geq c|x_d|.
    \end{align*}
    In the second estimate we used that $\mathcal F^{-1}\eta$ is a Schwartz function. On the other hand, the 
    theory of oscillatory integrals gives (see~\cite[p.361]{Stein1993}) 
    $$
      |\tilde K_j(x)| \les_M 2^{-j} (1+2^{-j}|x_d|)^{-M}  (1+|x_d|)^{-\frac{k}{2}}
      \qquad\text{for } |x'| \leq c|x_d|.
    $$
  \end{proof}  
   
  Next we use Proposition~\ref{prop:kernel} to find upper bounds for the operator norms of $\tilde T_j$ as
  maps from $L^p(\R^d)$ to $L^q(\R^d)$ where $1\leq p\leq q\leq \infty$.
  The latter condition is mandatory since $\tilde T_j$ is a
  translation-invariant operator covered by H\"ormander's result
  from~\cite[Theorem~1.1]{Hoermander_Translation}. 
  %\r{The following Lemma is visualized with the aid of a Riesz diagram, see Figure~\ref{fig:Riesz1}.}    

\begin{lem}\label{lem:Tdelta2}
  Assume $d\in \N,d\geq 2$ and let $\chi,\psi,k$ are as in~\eqref{eq:chipsi} and $\eta\in C_0^\infty(\R)$.
  Then, for any fixed $\eps>0$,  
  $$
    \|\tilde T_j\|_{p\to q} \les 2^{-j A_\eps(p,q)} 
     \qquad \text{for }1\leq p\leq  q\leq \infty,\;j\in\Z,j\geq j_0.
  $$
\end{lem}
\begin{proof}
  We first analyze the range $1\leq p\leq 2\leq q\leq \infty$.
  Plancherel's Theorem gives
  $$
    \|\tilde T_j f\|_2
    = \|\eta\left(2^j(\xi_d-\psi(\xi'))\right)\chi(\xi')\hat f\|_2
    \les \|\hat  f\|_2 = \|f\|_2
  $$
  due to $\eta,\chi \in L^\infty(\R^d)$. The Stein-Tomas Theorem for surfaces with $k$
  non-vanishing principal curvatures~\cite[p.365]{Stein1993} yields as in \cite[Lemma~4.3]{ManSch} 
  \begin{equation*}% \label{eq:SteinTomas}
    \|\tilde T_j f\|_q
    \les 2^{-\frac{j}{2}}\| f\|_2,\quad
    \|\tilde T_j f\|_2
    \les 2^{-\frac{j}{2}}\| f\|_{q'}
    \qquad\text{if } \frac{1}{q}\leq  \frac{k}{2(k+2)}.
  \end{equation*}
  The Restriction-Extension operator $f\mapsto \mathcal F^{-1}(\hat f\,d\sigma_M)$ 
  for compact pieces $M$ of hypersurfaces with $k$ non-vanishing principal curvatures
  has the mapping properties from \cite[Corollary~5.1]{ManSch}, so it is bounded for
  $(p,q)$ belonging to the pentagonal region 
  \begin{equation} \label{eq:RestExt}
    \frac{1}{p}>\frac{k+2}{2(k+1)},\quad \frac{1}{q}<\frac{k}{2(k+1)},\quad
    \frac{1}{p}-\frac{1}{q} \geq \frac{2}{k+2}.
  \end{equation}
  So for these exponents and $M_t:= \{\xi=(\xi',\xi_d)\in \supp(\chi)\times\R: \xi_d-\psi(\xi')=t\}$ with
  induced surface measure $d\sigma_{M_t}= (1+|\nabla\psi(\xi')|^2)^{1/2}\,d\xi'$ we  have for $\hat g(\xi):=\chi(\xi')\hat f(\xi)
  (1+|\nabla\psi(\xi')|^2)^{-1/2}$  
  $$
    \|\tilde T_j f\|_q
    \les \int_\R |\eta(2^jt)| \| \mathcal F^{-1}(\hat g\,d\sigma_{M_t}) \|_q\,dt
    \les \int_\R |\eta(2^jt)| \|g\|_p \,dt
    \les 2^{-j} \|f\|_p.
  $$
  Moreover, \cite[Corollary~5.1]{ManSch} yields restricted weak-type bounds from $L^{p,1}(\R^d)$ to
  $L^{q,\infty}(\R^d)$ for all $(p,q)$ belonging to the closure of the above-mentioned pentagon, which implies   
  $\|\tilde T_j f\|_{q,\infty} \les 2^{-j} \|f\|_{p,1}$ in the same manner.   
  Interpolating all these bounds gives 
  $$
    \|\tilde T_j\|_{p\to q}\les 2^{-j(\min\{A_0,A_1,A_2,A_2'\}-\eps\cdot\ind_{(p,q)\in\mathcal E})}
    =  2^{-jA_\eps(p,q)}
    \quad\text{for }1\leq p\leq 2\leq q\leq \infty,\,\eps>0.
  $$
  This finishes the analysis in the case $1\leq p\leq 2\leq q\leq \infty$. For  
  $2\leq p\leq q\leq \infty$ or $1\leq p\leq q\leq 2$ we get from Proposition~\ref{prop:kernel} 
  $$
    \|\tilde T_j\|_{1\to 1}+\|\tilde T_j\|_{\infty\to\infty}
    \les \|\tilde K_j\|_1 \les 2^{-j(\frac{k+2}{2}-d)}.
  $$
  Interpolating the estimates for $(p,q)=(\infty,\infty)$ with the ones for $p=2,q\geq 2$ from above yields
  the estimates in the region $A_3',A_4'$; the dual ones follow analogously. So we get 
  $$
    \|\tilde T_j\|_{p\to q}\les 2^{-j\min\{A_3,A_3',A_4,A_4'\}}=2^{-jA_\eps(p,q)},
  $$ 
  which proves the claim.
\end{proof}
  
  The optimality of our constants is open. It would be interesting to see
  whether recent results and techniques for oscillatory integral operators by Guth, Hickman,
  Iliopolou~\cite{GuthHickmanIliopoulou2019} or Kwon, Lee~\cite{KwonLee2020} (Proposition~2.4,
  Proposition~2.5) can be adapted to prove better bounds, especially  
  in the range $1\leq p\leq q<2$ or $2<p\leq q\leq \infty$. Any theorem leading to a larger value of
  $A_\eps(p,q)$ will automatically provide a larger range of exponents $q,r_1,r_2$ for which our
  Gagliardo-Nirenberg inequalities hold. Candidates for such values $\geq A_\eps(p,q)$
  are given in~\cite[Lemma~2.2]{ChoKimLeeShim2005} and \cite[Lemma~4.4]{ManSch}, but it seems nontrivial to
  make use of those in our setting. Next we use the estimates for $\tilde T_j$ to discuss the relevant
  operators at distance $2^{-j}$ from the critical surface where $j\nearrow +\infty$.
   
\begin{prop} \label{prop:DyadicEstimatesII}
  Assume $d\in\N$ and (A1) with $\alpha_1,\alpha_2>-1$. Then there are bounded linear operators $\mathcal
  T_j:L^p(\R^d)\to L^q(\R^d)$ and $j_0\in\Z$ with 
  $\sum_{j=j_0}^\infty \mathcal T_j u = u_1$ such that, for $i=1,2$ and any given $\eps>0$, we have for all  $u\in\mathcal S(\R^d)$, 
  $$
       \|\mathcal T_ju\|_{q} \les 2^{j (\alpha_i-A_\eps(p,q))} \|P_i(D)u\|_p.
        \qquad \text{for }1\leq p\leq  q\leq \infty,\;j\in\Z, j\geq j_0.
  $$
\end{prop}
\begin{proof}
   Recall $u_1= \mathcal F^{-1}(\tau\hat u)$ where $\tau$ was chosen in~\eqref{eq:u1u2}; we first consider
    the case $d\geq 2$. According to Assumption~(A1) there are $\tau_1,\ldots,\tau_L\in
    C_0^\infty(\R^d)$ such that $\tau_1+\ldots+\tau_L=\tau$ holds and $S\cap \supp(\tau_l) = \{ \xi \in\supp(\tau_l) : \tilde\xi_d=\psi_l(\tilde\xi') \text{ where
    }\tilde\xi=\Pi_l\xi\}$. Here, $\Pi_l$ denotes some permutation of coordinates in~$\R^d$. Since $P$
    vanishes of order $\alpha$ near the surface in the sense of Assumption (A1), we may write
  \begin{align} \label{eq:SymbolLocalization}
    \begin{aligned}
     &P(\xi)^{-1}\tau_l(\xi) =   \left[ \tau_{l+}(\xi)  
    (\tilde\xi_d-\psi_l(\tilde\xi'))_+^{-\alpha} 
    + \tau_{l-}(\xi)(\tilde\xi_d-\psi_l(\tilde\xi'))_-^{-\alpha} \right] \chi_l(\tilde\xi')  \\
    &\text{with}\quad
   \tau_{l+},\tau_{l-}\in C_0^\infty(\R^d),\; \chi_l\in C_0^\infty(\R^{d-1}),\; 
   \tilde\xi := \Pi_l\xi. 
  \end{aligned}
  \end{align}  
    for suitable functions $\chi_l,\psi_l$ that satisfy \eqref{eq:chipsi}. 
  In view of this we define 
  \begin{align*}
    \mathcal T_j := \sum_{l=1}^L \mathcal T_j^l
    \quad\text{where }
   \mathcal T_j^l u &:=  \mathcal F^{-1}\left( \tau_l(\xi)\hat u(\xi) \, 
   \eta(2^j(\tilde\xi_d-\psi_l(\tilde\xi'))) \chi_l(\tilde\xi')  \right)\quad (\tilde
   \xi=\Pi_l\xi).      
  \end{align*}
  Since $0$ does not belong to the support of $\eta$, there is $j_0\in\Z$ such that 
  $u_1= \sum_{j=j_0}^\infty \mathcal T_j u$ in the sense of distributions.
  We introduce the smooth function $\eta_{i}(z) := \eta(z) |z|^{-\alpha_i}$. Then
  Lemma~\ref{lem:Tdelta2} yields
 \begin{align*}
   \|\mathcal T_j u\|_q
   &\les \sum_{l=1}^L \|\mathcal T_j^l u\|_q \\
   &= \sum_{l=1}^L  \|\mathcal F^{-1}\left( \eta(2^j(\tilde\xi_d-\psi_l(\tilde\xi'))) \chi_l(\tilde\xi')\,
    \tau_l(\xi)\hat u(\xi) \right)\|_q  \\
   &= \sum_{l=1}^L \|\mathcal F^{-1}\left( 
   \eta(2^j(\tilde\xi_d-\psi_l(\tilde\xi'))) \chi_l(\tilde\xi') \, P_i(\xi)^{-1}\tau_l(\xi)
   \widehat{P_i(D)u}(\xi) \right)\|_q \\
   &\stackrel{\eqref{eq:SymbolLocalization}}= \sum_{l=1}^L  2^{j\alpha_i} \|\mathcal F^{-1}\left(
   \eta_{i}(2^j(\tilde\xi_d-\psi_l(\tilde\xi'))) \chi_l(\tilde\xi') 
   (\tau_{li+}(\xi)+\tau_{li-}(\xi))\widehat{P_i(D)u}(\xi)  \right)\|_{q}  \\
   &\les \sum_{l=1}^L  2^{j(\alpha_i-A_\eps(p,q))} \|\mathcal F^{-1}\left(
   (\tau_{li+}(\xi)+\tau_{li-}(\xi))\widehat{P_i(D)u}(\xi)\right)\|_p \\
   &\les 2^{j(\alpha_i-A_\eps(p,q))} \|P_i(D)u\|_p  
 \end{align*}
  In the last inequality we used that $\tau_{li+},\tau_{li-}$ are $L^p$-multipliers since their Fourier
  transforms are integrable. 
 \end{proof}

 In the forthcoming analysis we shall need the following auxiliary result.
 The proof mainly follows Stein's analysis of oscillatory integrals on \cite[p.380-386]{Stein1993}.

 \begin{prop} \label{prop:Lalpha}
   Assume $0\leq \alpha<\frac{1}{2}$ and that $\chi,\psi$ are as in \eqref{eq:chipsi}, 
   $\tau\in C_0^\infty(\R^d)$,  set 
   $$
     L_\alpha u:= \mathcal F^{-1}\left((\xi_d-\psi(\xi'))_+^{-\alpha} \chi(\xi')
     \tau(\xi) u\right).
   $$ 
   Then  $L_\alpha:L^2(\R^d)\to L^q(\R^d)$ is a bounded linear
   operator for $q:= \frac{2(k+2)}{k+2-4\alpha}$.
 \end{prop}
 \begin{proof}  
   Define the family of distributions $\gamma_s$ as in \cite[p.381]{Stein1993} 
   (called $\alpha_s$ in this book)  via
   $$
     \gamma_s(y) = \frac{e^{s^2}}{\Gamma(s)}y^{s-1} \zeta(y) 1_{y>0} 
     \qquad \text{if }\Re(s)>0.
   $$
   where $\zeta$ is smooth with compact support and $\zeta(y)=1$ for $|y|\leq y_0$ where $y_0$
   is chosen so large that $\zeta(\xi_d-\psi(\xi'))=1$ holds whenever $\chi(\xi')\tau(\xi)\neq 0$. 
   The family $(\gamma_s)$ is extended to all $s\in\C$ via analytic continuation.
   Then introduce the family of linear operators 
   $$
     M_s f:= \mathcal F^{-1}\left(\chi(\xi')^2 \gamma_s(\xi_d-\psi(\xi'))\hat f\right).
   $$
   Plancherel's Identity gives     
   $$
     \|M_s f\|_2 \les \|f\|_2 \qquad \text{if }\Re(s)=1.
   $$
   On the other hand 
   \begin{align*}
     M_s f
     = \Phi\ast f,\qquad
     \Phi(z):= \hat \gamma_s(-z_d) \cdot 
     \int_{\R^{d-1}} \chi(\xi')^2 e^{iz\cdot(\xi',\psi(\xi'))}\,d\xi' 
   \end{align*}
   From eq.~(15) in \cite{Stein1993} and eq.~(32) in \cite{ManSch} we infer
   $$
     |\Phi(z)|
     \les (1+|z_d|)^{-\Re(s)} (1+|z_d|)^{-\frac{k}{2}} 
     \les 1
     \qquad\text{if }\Re(s)=-\frac{k}{2}.
   $$
   We conclude
   $$
     \|M_s f\|_\infty \les \|f\|_1 \qquad \text{if }\Re(s)=-\frac{k}{2}.
   $$    
   Furthermore, for any given Schwartz functions $f,g$ the function $s\mapsto \int_{\R^d} (M_s f)g$ is
   holomorphic in the open strip $-\frac{k}{2}<\Re(s)<1$ with continuous extension to the boundary.
   So the family $M_s$ is admissible for Stein's Interpolation Theorem 
   \cite[Theorem~1]{Stein_Interpolation} and we obtain
   $$
     \|M_{1-2\alpha} f\|_q \les \|f\|_{q'} \qquad \text{if }
     \theta\in [0,1],\; 1-2\alpha = (1-\theta)\cdot (-\frac{k}{2})+\theta\cdot 1,\; 
     \frac{1}{q} = \frac{1-\theta}{\infty}+\frac{\theta}{2}.
   $$ 
   This leads to $\theta=\frac{2(k+2-4\alpha)}{2(k+2)}$ and $q= \frac{2(k+2)}{k+2-4\alpha}$.
   In view of $0<2\alpha<1$ this implies
   $$
     \|\mathcal F^{-1}\left(\chi(\xi')^2 (\xi_d-\psi(\xi'))_+^{-2\alpha}
     \zeta(\xi_d-\psi(\xi'))
     \hat f\right) \|_q \les \|f\|_{q'}. 
   $$
   Now we consider functions $\hat f = \tau^2 \hat g$. By choice of $\zeta$ and of $y_0$ we then have 
   $$
     \|\mathcal F^{-1}\left(\chi(\xi')^2 (\xi_d-\psi(\xi'))_+^{-2\alpha}
     \tau(\xi)^2 \hat g\right) \|_q 
     \les \|\mathcal F^{-1}(\tau^2 \hat g)\|_{q'}
     \les \|g\|_{q'}. 
   $$
   This implies the claim  given that that this operator coincides with $L_\alpha L_\alpha^*$. 
 \end{proof}
 
 We now use the dyadic estimates from Proposition~\ref{prop:DyadicEstimatesII} to prove 
 Gagliardo-Nirenberg inequalities   in the special case $P_1(D)=P_2(D)$ where the exponents satisfy
 $A_\eps(p,q)=\alpha\in [0,1]$. This result plays the same role in the critical frequency regime as
 Proposition~\ref{prop:BesselPotentials} does in the non-critical regime. For $d\geq 2$ we concentrate on exponents with $1\leq p\leq 2\leq q\leq \infty$.
  
  \begin{lem}\label{lem:SobolevIneq}
    Assume $d\in\N$ and let $P:=P_1=P_2$ satisfy (A1) for $\alpha:=\alpha_1=\alpha_2\in [0,1]$. Then 
    $\|u_1\|_q \les \|P(D)u\|_p$ holds for all $u\in\mathcal S(\R^d)$ provided that 
    \begin{itemize}
      \item[(i)] $d=1$ and $1\leq p,q\leq \infty$ satisfy $\frac{1}{p}-\frac{1}{q}=\alpha$ and, if
      $0<\alpha<1$, $(p,q)\notin \{(1,\frac{1}{1-\alpha}),(\frac{1}{\alpha},\infty)\}$, 
      \item[(ii)]  $d\geq 2$ and $1\leq p\leq 2\leq q\leq \infty$ satisfy    
      $\frac{1}{p}-\frac{1}{q}  =   \frac{2\alpha}{k+2}$ and $\min\{\frac{1}{p},\frac{1}{q'}\}
      > \frac{k+2\alpha}{2(k+1)}$.  
       %\r{$k=0$?}  %\r{$d\geq 2$: $(p,q)\in [1,2]\times [2,\infty]\sm\mathcal E_\alpha$}. 
    \end{itemize}
    The estimate $\|u_1\|_{q,\infty} \les \|P(D)u\|_p$ holds for exponents as in (i),(ii) or 
    \begin{itemize}
      \item[(iii)] $d=1,p=1,q=\frac{1}{1-\alpha}$ if $\alpha\in (0,1)$,
      \item[(iv)] $d\geq 2, 1\leq p<\frac{2(k+1)}{k+2\alpha},
      q=\frac{2(k+1)}{k+2-2\alpha}$ if $\alpha\in (\frac{1}{2},1]$. 
    \end{itemize}
  \end{lem}
  \begin{proof}
    With the same notations as before we have  
  \begin{align*}% \label{eq:SymbolLocalization}
    \begin{aligned}
     &P(\xi)^{-1}\tau_l(\xi) =   \left[ \tau_{l+}(\xi)  
    (\tilde\xi_d-\psi_l(\tilde\xi'))_+^{-\alpha} 
    + \tau_{l-}(\xi)(\tilde\xi_d-\psi_l(\tilde\xi'))_-^{-\alpha} \right] \chi_l(\tilde\xi')  \\
    &\text{with}\quad
   \tau_{l+},\tau_{l-}\in C_0^\infty(\R^d),\; \chi_l\in C_0^\infty(\R^{d-1}),\; 
   \tilde\xi := \Pi_l\xi. 
  \end{aligned}
  \end{align*}  
    for  functions $\chi_l,\psi_l$ that satisfy \eqref{eq:chipsi}. So $u_1=\sum_{j=j_0}^\infty
    \mathcal T_j u$. Assuming $1\leq p\leq 2\leq q\leq \infty$ are chosen as above we obtain (ii),(iv) as
    follows: 
    \begin{itemize}
      \item Case   $d\geq 2, \alpha=0$. \\ %~\smallskip
      Our assumptions give that $A_\eps(p,q)=\alpha=0$ only occurs for $p=q=2$. Here the
      estimate $\|u_1\|_2\les \|P(D)u\|_2$  follows from  Plancherel's Theorem.
      \item Case  $d\geq 2, \alpha\in (0,1)$. \\ %~\smallskip
       We first consider the case $\alpha<\frac{1}{2}$. By assumption, $(\frac{1}{p},\frac{1}{q})$ lies on
      the green diagonal line in Figure~\ref{fig:SobolevIneq}. By
      Proposition~\ref{prop:Lalpha}, the claimed inequality holds for the endpoints of that line given 
      by $p=2,q=\frac{2(k+2)}{k+2-4\alpha}$ and its dual $p=\frac{2(k+2)}{k+2+4\alpha},q=2$.
      Interpolating these two estimates with each other provides the desired inequality for all tuples on
      the green line in Figure~\ref{fig:SobolevIneq} and thus proves the claim for $\alpha<\frac{1}{2}$. \\
      Now consider the case $\alpha\geq \frac{1}{2}$. Our  assumptions imply that
      $(\frac{1}{p},\frac{1}{q})$ lies on the blue line in Figure~\ref{fig:SobolevIneq} with endpoints
      excluded. In particular, $(\frac{1}{p},\frac{1}{q})$ is in the interior of the $A_1$-region, so
      $A(\tilde p,\tilde q)=\frac{k+2}{2}(\frac{1}{\tilde p}-\frac{1}{\tilde q})$ for all $(\tilde p,\tilde
      q)$ close to $(p,q)$. For small $\delta>0$ we choose $\frac{1}{q_1} = \frac{1}{q}+\delta$,
      $\frac{1}{q_2} = \frac{1}{q}-\delta$. Interpolating the estimates for $(p,q_1)$ and $(p,q_2)$ with
      interpolation parameter $\theta=\frac{1}{2}$ gives, due to
      $(1-\theta) A_\eps(p,q_1)+\theta A_\eps(p,q_2)=\alpha$, the weak estimate $\|u\|_{q,\infty}\les
	  \|P(D)u\|_p$. Here we used $u_1=\sum_{j=j_0}^\infty  \mathcal T_j u$, the dyadic
	  estimates from Proposition~\ref{prop:DyadicEstimatesII} and the Interpolation Lemma~\ref{lem:SummationLemma}.
	  These weak estimates hold for all $(\frac{1}{p},\frac{1}{q})$ on the blue line with endpoints excluded.
	  Interpolating these inequalities with each other gives $\|u\|_q\les \|P(D)u\|_p$ for the same set of
	  exponents, which proves (ii) for $\alpha\in (0,1)$. \\ 
	  The prove the weak estimate from (iv) assume $\alpha\in (\frac{1}{2},1)$. For any given
	  $(\frac{1}{p},\frac{1}{q})$ on the dashed horizontal blue line in Figure~\ref{fig:SobolevIneq} with left
	  endpoint excluded we can choose $q_1,q_2$ as above and the same argument gives $\|u\|_{q,\infty}\les
	  \|P(D)u\|_p$. Since these exponents are given by $1\leq p<\frac{2(k+1)}{k+2\alpha}$ and
	  $q=\frac{2(k+1)}{k+2-2\alpha}$, we are done.  
      \item  Case  $d\geq 2, \alpha=1$. \\ % ~\smallskip
      It was shown in \cite[Section~5]{ManSch} that the linear operators $(P(D)+i\delta)^{-1}:L^p(\R^d)\to
      L^q(\R^d)$ are uniformly bounded with respect to small $|\delta|>0$ given that our additional
      regularity assumptions on $P$ from (A1) imply that $S=\{\xi\in\R^d: P(\xi)=0\}$ is a smooth compact manifold with $|\nabla P|\neq 0$ on $S$.
      This implies $\|u_1\|_q \les \|P(D)u\|_p$ and analogous arguments yield the weak bounds claimed in (iv).
    \end{itemize}

\begin{figure}[htbp]
%% d=4,k=2,N=3,\alpha\in\{0.25,0.75\} 
\begin{tikzpicture}[scale=10]
\draw[->] (0,0) -- (1.05,0) node[right]{$\frac{1}{p}$};
\draw[->] (0,0) -- (0,1.05) node[above]{$\frac{1}{q}$};
\draw (0,0) --(1,1);
\draw (1,1) -- (1,0) node[below]{$1$};
\draw (0,1) node[left]{$1$};
\coordinate (O) at (0,0);
\coordinate (O') at (1,1);
\coordinate (ST) at (0.5,0.25); %%$0.25=\frac{k}{2(k+2)}$
\coordinate (ST') at (0.75,0.5);
\coordinate (R) at (0.67,0.17); %% \frac{2}{3} = \frac{k+2}{2(k+1)}, \frac{1}{6} = \frac{k^2}{2(k+1)(k+2)}
\coordinate (R') at (0.83,0.33);
\coordinate (H) at (0.5,0); %%$0.25=\frac{k}{2(k+2)}$
\coordinate (H') at (1,0.5);
\coordinate (C) at (0.5,0.5);
\coordinate (E) at (1,0);
\coordinate (ER) at (0.67,0);
\coordinate (ER') at (1,0.33);
\draw (0.85,0.15) node [rounded corners=8pt]
{$A_0$};
\draw (0.65,0.4) node [rounded corners=8pt]
{$A_1$};
\draw (0.9,0.38) node [rounded corners=8pt]
{$A_2$};
\draw (0.62,0.1) node [rounded corners=8pt]
{$A_2'$};
\draw (0.95,0.75) node [rounded corners=8pt]
{$A_3$};
\draw (0.25,0.05) node [rounded corners=8pt]
{$A_3'$};
\draw (0.7,0.6) node [rounded corners=8pt]
{$A_4$};
\draw (0.4,0.27) node [rounded corners=8pt]
{$A_4'$};
%\draw [dotted] (0,0.5) node[left]{$\frac{1}{2}$} -- (0.5,0.5);
%\draw [dotted] (0,0.25) node[left]{$\frac{k}{2(k+2)}$} -- (ST);
%\draw [dotted] (0,0.17) node[left]{$\frac{k^2}{2(k+1)(k+2)}$} -- (R);
%\draw [dotted] (0.75,0) node[below]{\hspace{3mm}$\frac{k+4}{2(k+2)}$} -- (ST');
%\draw [dotted] (0.5,0) node[below]{$\frac{1}{2}$} -- (0.5,0.5);
\draw  (C) -- (H);
\draw  (C) -- (H');
\draw  (C) -- (O);
\draw   (C) -- (O');
\draw  (ST) -- (O);
\draw  (ST') -- (O');
\draw  (ST) -- (R);
\draw  (ST') -- (R'); 
\draw   (R) -- (R');
\draw  (R) -- (ER);
\draw     (R') -- (ER');
\draw [dotted] (0,0.375) node[left]{$\frac{k+2-4\alpha_2}{2(k+2)}$} -- (0.5,0.375);
\draw [dotted] (0,0.312) node[left]{$\frac{(d+1-2\alpha_2)k+2-4\alpha_2}{2d(k+2)}$} --
(0.5,0.312); 
%\draw [dotted] (0,0.55) node[left]{$\frac{2d-k-2\alpha_2}{2(2d-k-1)}$} -- (1,0.55);
\draw [line width = 0.4mm, draw=magenta]   (0.5,0.312) -- (0.688,0.5);
\draw [line width = 0.4mm, draw=green]  (0.5,0.375) -- (0.625,0.5);
\draw [line width = 0.4mm, draw=blue]   (0.792,0.417) -- (0.583,0.208);
\draw [dashed, line width = 0.4mm, draw=blue] (1,0.417) -- (0.792,0.417);
\draw [dashed, line width = 0.4mm, draw=blue] (0.583,0.208)  -- (0.583,0)
node[below]{\hspace{-3mm}$\frac{k+2\alpha_1}{2(k+1)}$};
\end{tikzpicture}
 \caption{ 
 Riesz diagram showing the exponents $1\leq p\leq 2\leq q\leq \infty$ 
 satisfying $A_\eps(p,q)=\alpha$ in the case $\alpha=\alpha_1\in (\frac{1}{2},1)$
 (blue) and for $\alpha=\alpha_2\in (0,\frac{1}{2})$ (green). 
 For the green resp. non-dashed blue exponent pairs Lemma~\ref{lem:SobolevIneq} (i),(ii) gives  $\|u\|_q\leq
 \|P(D)u\|_p$. In the case $\alpha=\alpha_2$ the corresponding estimates    
 from \cite[Theorem~1.4~(ii)]{ManSch} only hold for exponents on the magenta line.
 The picture was produced with parameter values
  $(d,k,\alpha_1,\alpha_2)=(4,2,\frac{3}{4},\frac{1}{4})$.}
   \label{fig:SobolevIneq}
\end{figure}
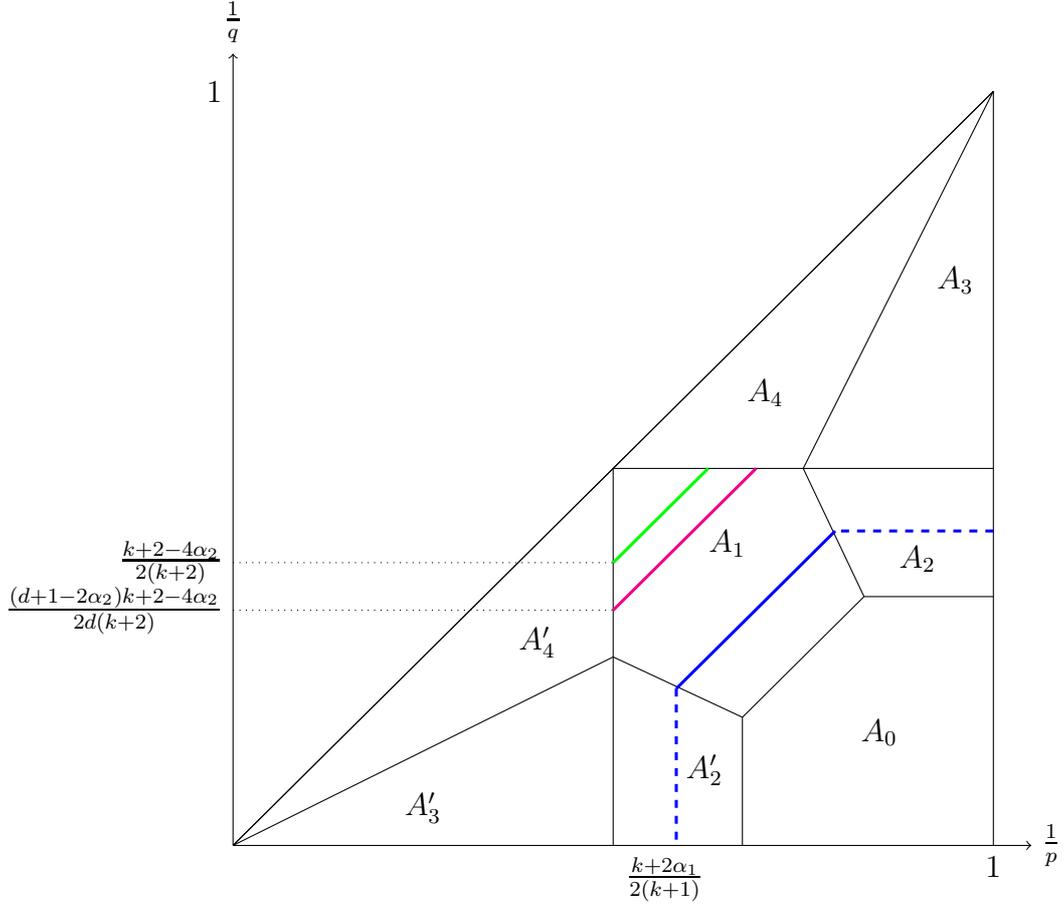
    
    \medskip
    
    Next we turn to the one-dimensional case $d=1$. The representation formula then reads
    \begin{align} \label{eq:formulau1}
     u_1 
     &= \sum_{l=1}^L \mathcal F^{-1}\left( \big[\tau_{l+}(\xi) (\xi-\xi_l^*)_+^{-\alpha} 
     +\tau_{l-}(\xi) (\xi-\xi_l^*)_-^{-\alpha}\big]\widehat{P(D)u}\right)  
   \end{align}
    where $\{P(\xi)=0\}=\{\xi^*_1,\ldots,\xi^*_L\}$. Using our assumption
    $\frac{1}{p}-\frac{1}{q}=\alpha$ we obtain the claims (i),(iii) from the following arguments:
    \begin{itemize}
      \item Case   $d=1, \alpha=0$. \\ 
      We then have $p=q$ and we first analyze $1<p=q<\infty$. In this case 
      the Hilbert transform $f\mapsto \mathcal F^{-1}(\sign(\xi)\hat f)$ is bounded on $L^p(\R)$, and so is
      $f\mapsto \mathcal F^{-1}(\sign(\xi-\xi_l^*)\hat f)$ for $l=1,\ldots,L$. 
      So the representation formula~\eqref{eq:formulau1} implies $\|u_1\|_p \les  \| P(D)u \|_p$.  
      In the case  $p=q\in \{1,\infty\}$ we  make use of our additional regularity assumption 
      $\tau_{l}:=\tau_{l+}=\tau_{l-}$ from (A1), so 
      \begin{align*}
        \|u_1\|_p
        \leq \sum_{l=1}^L \|\mathcal F^{-1} (\tau_{l} \widehat{P(D)u})\|_p 
        \les \sum_{l=1}^L \|  \mathcal F^{-1}(\tau_l) \ast (P(D)u) \|_p  
        \les \|P(D)u\|_p.   
      \end{align*}
      Here we used that $\mathcal F^{-1}(\tau_l)$ is a Schwartz function for $l=1,\ldots,L$. 
      \item Case $d=1,  \alpha\in (0,1)$ \\  
      If $1<p<q<\infty$ we deduce the claimed estimate from the boundedness of the Hilbert transform on
      $L^q(\R)$ and the Riesz potential estimate $\|\mathcal F^{-1}( |\cdot|^{-\alpha}\hat f)\|_q \les
      \|f\|_p$.
       For  $p=1,0<\alpha<1$ we have a weak estimate $\|\mathcal F^{-1}( |\cdot|^{-\alpha}\hat f)\|_{q,\infty}
       \les \|f\|_1$, see~\cite[Theorem 1.2.3]{Graf_Modern}. Note that the Hilbert transform
       is bounded on $L^{q,\infty}(\R)$ as well by real interpolation.  
      \item Case   $d=1,\alpha=1$. \\ 
      We now have $\frac{1}{p}-\frac{1}{q}=1$, so $p=1,q=\infty$. We exploit  
      the additional smoothness assumption $\tau_{l+}=-\tau_{l-}$ from (A1). Then       
      $P\in C^\infty(\R)$ is a smooth function with simple zeros $\xi_1^*,\ldots,\xi_L^*$. To prove the
      claimed inequality we start with the trivial estimate  $\|v\|_\infty\les \|v'\|_1 = \|\mathcal
      F^{-1}(i\xi \hat v)\|_1$ for all $v\in\mathcal S(\R)$. Translation in Fourier space gives
      $\|v\|_\infty\les \|  \mathcal F^{-1}(i(\xi-\xi_l^*)\hat v) \|_1$  for all $u\in\mathcal
      S(\R),l=1,\ldots,L$.
      %Since the zeros $\xi_l^*$ are simple and their union equals $\{\xi\in\R: P(\xi)=0\}$,
      So \eqref{eq:formulau1} implies as above  
      \begin{align*}
        \|u_1\|_\infty
        &\les \sum_{l=1}^L \|\mathcal F^{-1}((\xi-\xi_l^*)^{-1}\tau_l \widehat{P(D)u})\|_\infty\\
        &\les \sum_{l=1}^L \|  \mathcal F^{-1}(\tau_l\widehat{P(D)u}) \|_1  \\
        &\les \|P(D)u\|_1.
      \end{align*}      
      %where $m(\xi) = \tau(\xi)(\xi-\xi^*)P(\xi)^{-1}$ satisfies the assumptions of
      %Proposition~\ref{prop:multiplier}.
    \end{itemize}
    %\medskip
 \end{proof}

  As remarked in Figure~\ref{fig:SobolevIneq}, claim (ii) of the previous lemma improves upon the
  corresponding bounds from \cite[Theorem~1.4]{ManSch} in the case $0<\alpha<\frac{1}{2}$.   
  We finally combine all these estimates to prove Gagliardo-Nirenberg inequalities in the critical frequency
  regime. Given the rather complicated definition of $A_\eps(p,q)$, an explicit characterization of the admissible
  exponents is possible in principle, but extremely laborious. We prefer to avoid most of the 
  computations. Instead, we describe the set of admissible exponents in an abstract way and provide the
  required computations in the reasonably simple special case $1\leq p\leq 2\leq q\leq \infty$ that allows to
  prove our main results.  Proceeding in this way it becomes clear, how eventual improvements of
  Lemma~\ref{lem:Tdelta2} affect the final range of exponents. Once more we exploit Bourgain's summation
  argument, which allows us to argue almost as in the large frequency regime. On a formal level, comparing
  Lemma~\ref{lem:Tdelta} (large frequencies) with Lemma~\ref{lem:Tdelta2} (critical frequencies),  
  we essentially have to replace $s_i-d(\frac{1}{r_i}-\frac{1}{q_i})$ by $A_\eps(r_i,q_i)-\alpha_i$ because
  the summation index now ranges from some $j=j_0$ to $+\infty$ and not from $j=j_0$ to $-\infty$.
  It will be convenient to formulate our sufficient conditions in terms of $\ov \alpha:=
  (1-\kappa)\alpha_1+\kappa\alpha_2$.  %that we will assume to satisfy $\ov\alpha\in [0,1]$.
  
  \medskip
  
  We provide a definition of the set $\mathcal A(\kappa)$ of exponents $(q,r_1,r_2)$ that are
  admissible for
  \begin{equation}\label{eq:GN_u1}
    \|u_1\|_q \les \|P_1(D)u\|_{r_1}^{1-\kappa}\|P_2(D)u\|_{r_2}^\kappa 
    \qquad (u\in\mathcal S(\R^d)).
  \end{equation}
  Lemma~\ref{lem:SobolevIneq} provides the definition for $\kappa\in\{0,1\}$, namely
  \begin{align} \label{eq:A(0)}
    \begin{aligned}
    \mathcal A(0) &:=  \Big\{ (q,r_1,r_2)\in [1,\infty]^3:  (q,r_1,\alpha_1) \text{ as in
    Lemma~\ref{lem:SobolevIneq}~(i),(ii)} \Big\}, \\
    \mathcal A(1) &:=  \Big\{ (q,r_1,r_2)\in [1,\infty]^3:  (q,r_2,\alpha_2) \text{ as in
    Lemma~\ref{lem:SobolevIneq}~(i),(ii)} \Big\}.
  \end{aligned}
  \end{align}
  In the case $0<\kappa<1$ the definition is more involved and relies on the 
  Interpolation Lemma~\ref{lem:SummationLemma} and the dyadic estimates for critical frequencies from
  Proposition~\ref{prop:DyadicEstimatesII}.
  Combining the latter with~\eqref{eq:SummationII} we obtain $\|u_1\|_q\les \|u\|_{(X_1,X_2)_{\kappa,q}}$
  and deduce \eqref{eq:GN_u1}  for exponents $(q,r_1,r_2)$ belonging to the set  
 \begin{align*}
     \mathcal A_1(\kappa)
     &:= \Big\{ (q,r_1,r_2)\in [1,\infty]^3: \text{ There are }
     \eps>0,\; q_1\in [r_1,\infty],\; q_2\in [r_2,\infty],  \text{ such that} \\ 
      &\qquad \frac{1}{q}=\frac{1-\kappa}{q_1}+\frac{\kappa}{q_2} \text{ and }
         (1-\kappa) A_\eps(r_1,q_1)+   \kappa  A_\eps(r_2,q_2)> \ov\alpha \Big\}. 
    \intertext{This result covers   all non-endpoint cases in our considerations
    further below. Using~\eqref{eq:SummationI} with $Y_1=Y_2=L^q(\R^d)$ we obtain $\|u_1\|_q
    %\les \|u\|_{(X_1,X_2)_{\kappa,1}}
    \les \|P_1(D)u\|_{r_1}^{1-\kappa}\|P_2(D)u\|_{r_2}^\kappa$  for   exponents in} 
     \mathcal A_2(\kappa)
     &:= \Big\{ (q,r_1,r_2)\in [1,\infty]^3:  q\geq \max\{r_1,r_2\} \text{ and there is }
     \eps>0 \text{ such that }\\
     &\qquad  (1-\kappa) A_\eps(r_1,q) + \kappa  A_\eps(r_2,q)= \ov\alpha,\,
        A_\eps(r_i,q)\neq \alpha_i \,(i=1,2)
     \Big\}. 
    \intertext{
    Next we use $\|u\|_q = \|u\|_q^{1-\kappa}\|u\|_q^\kappa$ to deduce further
    estimates from Lemma~\ref{lem:SobolevIneq}  for exponents in } 
    \mathcal A_3(\kappa)
     &:= \Big\{ (q,r_1,r_2)\in [1,\infty]^3:  
     (q,r_1,\alpha_1),\,(q,r_2,\alpha_2) \text{ as in Lemma~\ref{lem:SobolevIneq}~(i),(ii)} \Big\}.
%    \intertext{The second approach uses the weak estimtes from Lemma~\ref{lem:SobolevIneq}~(i)-(iv):
%    }
%    \mathcal A_4(\kappa)
%      &:= \Big\{ (q,r_1,r_2)\in [1,\infty]^3: \text{ There are }
%       q_1\in [r_1,\infty],\; q_2\in [r_2,\infty],  \text{ with } q_1\neq q_2
%        \text{ and } \\  
%      &\qquad  \frac{1}{q}=\frac{1-\kappa}{q_1}+\frac{\kappa}{q_2},\; (r_i,q_i,\alpha_i) \text{ as in
%      Lemma~\ref{lem:SobolevIneq}~(i)-(iv).} \Big\}.
    \intertext{Using \eqref{eq:SummationI}
     with $Y_1=L^{q_1}(\R^d),Y_2=L^{q_2}(\R^d)$    we   get the weak bound
    $\|u_1\|_{q,\infty} \les \|u\|_{(X_1,X_2)_{\kappa,1}}$ for   exponents belonging to}
     \mathcal A_4^w(\kappa)
      &:= \Big\{ (q,r_1,r_2)\in [1,\infty]^3: \text{ There are }\eps>0,\;
     q_1\in [r_1,\infty],\; q_2\in [r_2,\infty] \text{ such that} \\
       &\qquad (1-\kappa) A_\eps(r_1,q_1) + \kappa  A_\eps(r_2,q_2)=
       \ov\alpha,\; \frac{1}{q}=\frac{1-\kappa}{q_1}+\frac{\kappa}{q_2},\; \alpha_i \neq
       A_\eps(r_i,q_i),\;q_1\neq q_2 \Big\}.
     \intertext{Interpolating the (weak or strong) endpoint estimates 
     for $\mathcal A_2(\kappa)\cup\mathcal A_3(\kappa)\cup \mathcal A_4^w(\kappa)$
     with each other exactly as in the final step of the proof of of
     Proposition~\ref{prop:LargeFreqInterpolated} we deduce $\|u_1\|_q \les \|u\|_{(X_1,X_2)_{\kappa,q}} \les \|P_1(D)u\|_{r_1}^{1-\kappa}\|P_2(D)u\|_{r_2}^\kappa
     $ for exponents from}
     \mathcal A_4(\kappa)
    &:= \Big\{ (q,r_1,r_2)\in [1,\infty]^3: \text{ There are } \eps\neq 0,\delta>0,\, \tilde q,q^*\in
    [1,\infty],\,\tilde\kappa,\kappa^*\in (0,1) \text{ with }    \\
      &\qquad \frac{1}{\tilde q}-\eps = \frac{1}{q} = \frac{1}{q^*}+\eps,\,
      \tilde\kappa-\delta=\kappa=\kappa^*+\delta
      \text{ and }  \\
    &\qquad (\tilde q,r_1,r_2) \in\mathcal A^w_4(\tilde\kappa)\cup\mathcal
      A_3(\tilde\kappa)\cup\mathcal A_2(\tilde\kappa),\; (q^*,r_1,r_2) \in\mathcal A^w_4(\kappa^*)\cup\mathcal
      A_3(\kappa^*)\cup\mathcal A_2(\kappa^*) \Big\}.
%     \r{Finally, exploiting the fact that the critical frequencies lie in a bounded set ($S$ is compact)
%     we can even extend the range of admissible exponents to the set}
%     \mathcal A(\kappa)
%     &:= \Big\{ (q,r_1,r_2)\in [1,\infty]^3: \text{ There are }  \tilde q\leq q,\, \tilde r_1\geq
%     r_1,\,\tilde r_2\geq r_2  \text{ such that }   \\
%     &\qquad (\tilde q,\tilde r_1,\tilde r_2)\in\mathcal A_0(\kappa)\cup\mathcal A_1(\kappa)\cup \mathcal
%     A_2(\kappa)\cup\mathcal A_3(\kappa)\Big\}. \\
  \end{align*}

  Summarizing these interpolation results we obtain the following interpolation inequality in the critical
  frequency regime.
  %Note that the exponents from this set   always satisfy $1<q<\infty$.
   
\begin{prop} \label{prop:CritFreqInterpolated}
  Assume $d\in\N, \kappa\in [0,1]$ and (A1) for $\alpha_1,\alpha_2>-1$. Then
  $$
     \| u_1 \|_q  \les \|P_1(D)u\|_{r_1}^{1-\kappa} \|P_2(D)u\|_{r_2}^\kappa
     \qquad (u\in\mathcal S(\R^d))
  $$
  holds provided that $(q,r_1,r_2)\in\mathcal A(\kappa):= \mathcal A_1(\kappa)\cup\mathcal A_2(\kappa)\cup
  \mathcal A_3(\kappa)\cup\mathcal A_4(\kappa)$.
\end{prop}

\section{Gagliardo-Nirenberg Inequalities, Proofs of
Theorem~\ref{thm:GNhigherDspecial} and Theorem~\ref{thm:GN1Dspecial}.}

  We first discuss the one-dimensional case. As before, we use the notation 
  $$ 
    \ov\alpha:=(1-\kappa)\alpha_1+\kappa\alpha_2
    \qquad\text{and}\qquad \ov s:= (1-\kappa)s_1+\kappa s_2.
  $$
%   \begin{prop}
%      Assume $d=1,\kappa\in (0,1)$ and that (A1) for $\alpha_1=s_1\in\R,\alpha_1,\alpha_2>-1$. Then 
%   $$
%     \|u\|_q \les \|P_1(D)u\|_{r_1}^{1-\kappa} \|P_2(D)u\|_{r_2}^{\kappa}
%     \qquad (u\in\mathcal S(\R))
%   $$
%   holds provided that $q,r_1,r_2\in [1,\infty]$  satisfy $\ov\alpha_+\leq
%   \frac{1-\kappa}{r_1}+\frac{\kappa}{r_2}-\frac{1}{q} \leq \ov s$ subject to the following conditions:
%   \end{prop}

\begin{thm} \label{thm:GN1D}
  Assume $d=1,\kappa\in [0,1]$ and that (A1),(A2) hold for $s_1,s_2\in\R$ and $\alpha_1,\alpha_2>-1$ such
  that $0<\ov\alpha\leq \ov s$. Then 
  $$
    \|u\|_q \les \|P_1(D)u\|_{r_1}^{1-\kappa} \|P_2(D)u\|_{r_2}^{\kappa}
    \qquad (u\in\mathcal S(\R))
  $$
  holds provided that $q,r_1,r_2\in [1,\infty]$ satisfy 
  $\ov\alpha\leq \frac{1-\kappa}{r_1}+\frac{\kappa}{r_2}-\frac{1}{q} \leq \ov s$ 
  as well as the conditions (i),(ii),(iii) and (iv),(v),(vi) in the endpoint cases
  $\frac{1-\kappa}{r_1}+\frac{\kappa}{r_2}-\frac{1}{q}=\ov s$ and
  $\ov\alpha=\frac{1-\kappa}{r_1}+\frac{\kappa}{r_2}-\frac{1}{q}$, respectively:
  \begin{itemize}
    \item[(i)] if $q=\infty$ then $\frac{1}{r_1}- s_1\neq 0 \neq \frac{1}{r_2}-s_2$ or
    $(r_1,r_2)=(\frac{1}{s_1},\frac{1}{s_2}), s_1,s_2\in \{0,1\}$,
    \item[(ii)] if $1<q<\infty, \frac{1}{r_1}-\frac{s_1}{d}=\frac{1}{q}=\frac{1}{r_2}-\frac{s_2}{d}$ and
    $r_1=1$  then \\ $1<r_2<q,\, \kappa\geq \frac{r_2}{q}$ or $r_2=\infty,\frac{1}{q}\leq \kappa\leq
    \frac{1}{q'}$,
    \item[(iii)] if $1<q<\infty$ and $\frac{1}{r_1}-\frac{s_1}{d}=\frac{1}{q}=\frac{1}{r_2}-\frac{s_2}{d}$
     and $r_2=1$  then \\ $1<r_1<q,\,1-\kappa\geq \frac{r_1}{q}$ or $r_1=\infty,\frac{1}{q}\leq
     1-\kappa\leq \frac{1}{q'}$,
    \item[(iv)] if $q=\infty$ then $\frac{1}{r_1}-\alpha_1\neq 0\neq \frac{1}{r_2}-\alpha_2$ or
    $(r_1,r_2)=(\frac{1}{\alpha_1},\frac{1}{\alpha_2}), \alpha_1,\alpha_2\in \{0,1\}$,
    \item[(v)] 
    %if $1<q<\infty, \frac{1}{r_1}-\alpha_1=\frac{1}{q} = \frac{1}{r_2}-\alpha_2,
    %\min\{r_1,r_2\}=1$ then $1<\max\{r_1,r_2\}<q,\, \kappa\geq \frac{\max\{r_1,r_2\}}{q}$.
    %\r{
    if $1<q<\infty, \frac{1}{r_1}-\alpha_1=\frac{1}{q} = \frac{1}{r_2}-\alpha_2$ 
    then \\
    $\alpha_1,\alpha_2\in [0,1]$ and $r_1=1,\kappa<1$ only if $1<r_2<q,\, \kappa\geq \frac{r_2}{q}$,
    %}
    \item[(vi)] 
    %if $1<q<\infty, \frac{1}{r_1}-\alpha_1=\frac{1}{q} = \frac{1}{r_2}-\alpha_2,
    %\min\{r_1,r_2\}=1$ then $1<\max\{r_1,r_2\}<q,\, \kappa\geq \frac{\max\{r_1,r_2\}}{q}$.
    %\r{
    if $1<q<\infty, \frac{1}{r_1}-\alpha_1=\frac{1}{q} = \frac{1}{r_2}-\alpha_2$ 
    then \\
    $\alpha_1,\alpha_2\in [0,1]$ and $r_2=1,\kappa>0$ only if $1<r_1<q,\, 1-\kappa\geq \frac{r_1}{q}$.
    %}
  \end{itemize}
%   If $P_1,P_2$ are smooth and regular near $S$ then (iv) may be replaced by
%   \begin{itemize}
%     \item[(iv)']  $(q,\alpha_1,\alpha_2)\neq (\infty,\frac{1}{r_1},\frac{1}{r_2})$ if $1<r_i<\infty$ 
%     for some $i\in\{1,2\}$. 
%   \end{itemize}   
\end{thm}
\begin{proof}
  Proposition~\ref{prop:LargeFreqInterpolated} shows that the large frequency part of the inequality
  (involving $s_1,s_2$ and thus (i),(ii),(iii)) holds. In view of Proposition~\ref{prop:CritFreqInterpolated}
  it remains to show that all exponents satisfying  
  $\ov\alpha\leq \frac{1-\kappa}{r_1}+\frac{\kappa}{r_2}-\frac{1}{q}$ with (iv),(v),(vi) in the endpoint case
  $\ov\alpha= \frac{1-\kappa}{r_1}+\frac{\kappa}{r_2}-\frac{1}{q}$ are covered by $\mathcal A(\kappa)$. 
  In the case $\kappa=0$ this holds by definition of $\mathcal A(0)$ from~\eqref{eq:A(0)}  
  because the requirement  $(r_1,q)\notin\{1,\frac{1}{1-\alpha},\frac{1}{\alpha},\infty\}$ if $0<\alpha<1$
  from Lemma~\ref{lem:SobolevIneq}~(i) is met by (iv),(v),(vi). The discussion for $\kappa=1$ is analogous.
  So from now on consider the case $0<\kappa<1$.
  
  \medskip
  
  We now retrieve some information about $\mathcal A(\kappa)$ by exploiting the formula
  $A_\eps(p,q)=\frac{1}{p}-\frac{1}{q}$ for $1\leq p\leq q\leq\infty$, see~\eqref{eq:def_Aeps}. 
  Going back to the definition of the sets $\mathcal A_i(\kappa)$ we find 
  \begin{align*}
     \mathcal A_1(\kappa)
     &= \Big\{ (q,r_1,r_2)\in [1,\infty]^3:\,  
     \frac{1-\kappa}{r_1}+\frac{\kappa}{r_2}-\frac{1}{q}>\ov\alpha  \Big\},   \\
     \mathcal A_2(\kappa)
     &\supset \Big\{ (q,r_1,r_2)\in [1,\infty]^3:\,  
      \frac{1-\kappa}{r_1}+\frac{\kappa}{r_2}-\frac{1}{q}=\ov\alpha,\,
      0\leq \frac{1}{r_i}-\frac{1}{q}\neq \alpha_i \text{ for }i=1,2\Big\}, \\
    \mathcal A_3(\kappa)
     &\supset \Big\{ (q,r_1,r_2)\in [1,\infty]^3:    
     \frac{1-\kappa}{r_1}+\frac{\kappa}{r_2}-\frac{1}{q}=  \ov\alpha,\,
     \frac{1}{r_i}-\frac{1}{q}=\alpha_i\in [0,1]
     \text{ and } \\
     &\qquad (r_i,q)\notin
     \big\{\big(1,\frac{1}{1-\alpha_i}\big),\big(\frac{1}{\alpha_i},\infty\big)\big\} \text{ if
     }\alpha_i\in (0,1) \text{ for }i=1,2\Big\}.
   \end{align*}
   Since the interpolation inequality holds for these exponents, our claim is proved in the following cases:
   \begin{itemize}
     \item $\frac{1-\kappa}{r_1}+\frac{\kappa}{r_2}-\frac{1}{q}>\ov\alpha$: see $\mathcal A_1(\kappa)$.
     \item $\frac{1-\kappa}{r_1}+\frac{\kappa}{r_2}-\frac{1}{q}=\ov\alpha$ and $q=1$:  
     we necessarily have $\ov \alpha=0,r_1=r_2=1$, which is covered by $\mathcal A_2(\kappa)$
     for $\alpha_1,\alpha_2\neq 0$ or $\mathcal A_3(\kappa)$ for $\alpha_1=\alpha_2=0$, respectively. 
     \item $\frac{1-\kappa}{r_1}+\frac{\kappa}{r_2}-\frac{1}{q}=\ov\alpha$ and $q=\infty$: 
     $\frac{1}{r_1}-\alpha_1\neq 0\neq \frac{1}{r_2}-\alpha_2$ is covered by $\mathcal A_2(\kappa)$
     and $\frac{1}{r_1}-\alpha_1=0=\frac{1}{r_2}-\alpha_2$ with $\alpha_1,\alpha_2\in\{0,1\}$ is covered by
     $\mathcal A_3(\kappa)$. 
   \end{itemize}
   So it remains to show the remaining endpoint estimates dealing with $1<q<\infty$. By definition of
   $\mathcal A_4^w(\kappa)$ we have restricted weak-type estimates  for exponents from
   \begin{align*}
    \mathcal A_4^w(\kappa)
     &= \Big\{ (q,r_1,r_2)\in [1,\infty]^3:\, 
       \frac{1-\kappa}{r_1}+\frac{\kappa}{r_2}-\frac{1}{q}= \ov\alpha\text{ and there are }
       q_1\in [r_1,\infty], q_2\in [r_2,\infty]  \\
     &\qquad \text{ such that }
       q_1\neq q_2,\; \frac{1}{r_i}-\frac{1}{q_i}\neq \alpha_i \,(i=1,2),\; 
       \frac{1-\kappa}{q_1}+\frac{\kappa}{q_2} = \frac{1}{q}  
       \Big\}
       \\
     &= \Big\{ (q,r_1,r_2)\in [1,\infty]^3:\, 
       \frac{1-\kappa}{r_1}+\frac{\kappa}{r_2}-\frac{1}{q}= \ov\alpha,\;
       1<q<\infty  \Big\}.
   \end{align*}
   %In the case $\ov\alpha=0,\,\alpha_1,\alpha_2\neq 0, r_1\neq r_2$ one has to choose $(q_1,q_2)=(r_1,r_2)$
   (Indeed, thanks to $\ov \alpha>0$ we may choose $\frac{1}{q_1} := \frac{1}{r_1}-\eps$ and
   $\frac{\kappa}{q_2} := \frac{1}{q}-\frac{1-\kappa}{q_1}$ for small $\eps>0$ provided that $1\leq
   r_1<\infty$, analogously for $r_2<\infty$.) This implies    
   $$
     \mathcal A_4(\kappa)
      \supset \Big\{ (q,r_1,r_2)\in [1,\infty]^3:\,
      \frac{1-\kappa}{r_1}+\frac{\kappa}{r_2}-\frac{1}{q} =  \ov\alpha,\, 
      1<q<\infty, \, \frac{1}{r_1}-\frac{1}{r_2} \neq \alpha_1-\alpha_2 \Big\}.
  $$
  This yields the claim for the following exponents:    
  \begin{itemize}
    \item $\frac{1-\kappa}{r_1}+\frac{\kappa}{r_2}-\frac{1}{q}=\ov\alpha,\,1<q<\infty$ and
    $\frac{1}{r_1}-\frac{1}{r_2}\neq \alpha_1-\alpha_2$, which is covered by $\mathcal A_4(\kappa)$,
    \item $\frac{1-\kappa}{r_1}+\frac{\kappa}{r_2}-\frac{1}{q}=\ov\alpha,\,1<q<\infty$ and 
    $\frac{1}{r_i}-\frac{1}{q} = \alpha_i\in [0,1]$ with 
    $(r_i,q)\neq (1,\frac{1}{1-\alpha_i})$ if $\alpha_i\in (0,1)$, which is covered by $\mathcal A_3(\kappa)$.
  \end{itemize}
  So it remains to prove the claim for 
  \begin{align*}
    &1<q<\infty,\;\frac{1}{r_1}-\alpha_1=\frac{1}{q} =
   \frac{1}{r_2}-\alpha_2 \qquad\text{and} \\
    &\big[\; r_1=1<r_2<q,\; 1>\kappa\geq \frac{r_2}{q} \quad\text{or}\quad  r_2=1<r_1<q,\; 1>1-\kappa\geq
    \frac{r_1}{q}\; \big].
  \end{align*}
  By symmetry we may concentrate on $r_1=1<r_2<q, 1> \kappa\geq
  \frac{r_2}{q}$ where the estimate follows from 
  $$
    \|u\|_q 
    \stackrel{\eqref{eq:InterpolationIdentity}}\les  \|u\|_{q,\infty}^{1-\kappa} \|u\|_{q,\kappa
    q}^{\kappa} \les \|u\|_{q,\infty}^{1-\kappa} \|u\|_{q,r_2}^{\kappa}
    \les \|P_1(D)u\|_1^{1-\kappa} \|P_2(D)u\|_{r_2}^{\kappa}.
  $$  
  Here we used Proposition~\ref{prop:BesselPotentials}~(iv) and ~(ii) (for $r=r_2$).   
  This finishes the proof.
\end{proof}

 \medskip
 
 \noindent\textbf{Proof of Theorem~\ref{thm:GN1Dspecial}:} We apply Theorem~\ref{thm:GN1D}   
 to the symbols $P_1(D)=|D|^s-1,s>0$ and  $P_2(D)=I$ that  satisfy the hypotheses of the
 Theorem for $(\alpha_1,\alpha_2,s_1,s_2)=(1,0,s,0)$. Then $\ov\alpha=1-\kappa,\ov s= (1-\kappa)s$, so 
Theorem~\ref{thm:GN1D} implies that the Gagliardo-Nirenberg Inequality holds provided that $1-\kappa\leq
\frac{1-\kappa}{r_1}+\frac{\kappa}{r_2}-\frac{1}{q}\leq (1-\kappa)s$.  
The latter restriction comes from Theorem~\ref{thm:GN1D}~(i) and one checks that (ii)-(vi) are not restrictive
for our choice of parameters $(\alpha_1,\alpha_2,s_1,s_2)=(1,0,s,0),s>0$.   \qed
 
  \medskip

  We continue with the higher-dimensional case where a computation of $\mathcal A(\kappa)\cap \mathcal
  B(\kappa)$ is rather cumbersome. To simplify the discussion we concentrate on the special case
  $r_1=r_2=r\in [1,2]$ and $q\in [2,\infty]$ and only consider the special ansatz $q_1=q_2=q$ in the
  definition of the sets $\mathcal A_i(\kappa)$.
  %\r{Do we need the weak estimates?}

\begin{thm}\label{thm:GNhigherD}
  Assume $d\in\N,d\geq 2, \kappa\in [0,1]$ and that (A1),(A2) hold for $s_1,s_2\in\R$ and
  $\alpha_1,\alpha_2>-1$ such that $0\leq \ov\alpha\leq 1$.
  Then 
  $$
    \|u\|_q \les \|P_1(D)u\|_r^{1-\kappa} \|P_2(D)u\|_r^{\kappa}
    \qquad (u\in\mathcal S(\R^d))
  $$
  holds provided that  $\ov\alpha<1$, $\alpha_1\neq\alpha_2$, $0<\kappa<1$ and   
  the exponents $r\in [1,2],q\in [2,\infty]$ satisfy
  \begin{equation}\label{eq:GNconditionsTHM}
   \frac{2\ov\alpha}{k+2} 
    \leq \frac{1}{r}-\frac{1}{q} 
    \leq  \frac{\ov s}{d}
    \qquad\text{and}\qquad
     \min\left\{\frac{1}{r},\frac{1}{q'}\right\} 
       \geq \frac{k+2\ov\alpha}{2(k+1)} 
  \end{equation}
  as well as $(q,r)\neq (\infty,\frac{d}{\ov s})$ if $s_1=s_2= \ov s\in (0,d]$. In the case 
  $\ov\alpha=1$ or $\alpha_1=\alpha_2$ or $\kappa\in\{0,1\}$ the same is true provided that the last condition
  in \eqref{eq:GNconditionsTHM} is replaced by $\min\{\frac{1}{r},\frac{1}{q'}\} >
  \frac{k+2\ov\alpha}{2(k+1)}$. 
\end{thm}
\begin{proof}
  The conditions for large frequencies (involving $s_1,s_2$) were shown to be sufficient in
  Proposition~\ref{prop:LargeFreqInterpolated}. So we concentrate on the critical frequency part involving
  $\alpha_1,\alpha_2$. The following computations are based on the formula
  $A_\eps(r,q) = A(r,q)-\eps\cdot \ind_{(p,q)\in\mathcal E}$ where  
  $$
    A(r,q)=\min\left\{1,\frac{k+2}{2}\left(\frac{1}{r}-\frac{1}{q}\right),\frac{k+2}{2}-
    \frac{k+1}{q},-\frac{k}{2}+\frac{k+1}{r}\right\} 
  $$
  for $1\leq r\leq 2\leq q\leq \infty$, see \eqref{eq:def_Aeps} and Figure~\ref{fig:Tdeltabounds}. 
  Our definitions of $\mathcal A_1(\kappa),\mathcal A_2(\kappa),\mathcal A_3(\kappa)$ yield in the case
  $0<\kappa<1$ 
  \begin{align*} 
    \mathcal A_1(\kappa)
    &\supset \{(q,r,r)\in [2,\infty]\times [1,2]^2 : A_\eps(r,q)> \ov\alpha \text{ for some }\eps>0\},\\
    \mathcal A_2(\kappa)
    &\supset \{(q,r,r)\in [2,\infty]\times [1,2]^2 : A_\eps(r,q)= \ov\alpha \text{ for some }\eps>0,\,
    \alpha_1\neq \ov\alpha\neq\alpha_2\}, \\
    \mathcal A_3(\kappa)
    &\supset \Big\{(q,r,r)\in [2,\infty]\times [1,2]^2 : A_\eps(r,q)= \ov\alpha \text{ for some }\eps>0,\, 
    \alpha_1=\ov\alpha=\alpha_2\in [0,1]  \\
    &\hspace{5.5cm} \text{and
    }\min\left\{\frac{1}{r},\frac{1}{q'}\right\}> \frac{k+2\ov\alpha}{2(k+1)}\Big\}.
  \end{align*}
  From $\mathcal A(\kappa)\supset \mathcal A_1(\kappa)\cup \mathcal A_2(\kappa)\cup \mathcal A_3(\kappa)$
  we thus get 
  \begin{align*}
    \mathcal A(\kappa) 
    &\supset \Big\{(q,r,r)\in [2,\infty]\times [1,2]^2 : A_\eps(r,q)\geq  \ov\alpha \text{ for some
    }\eps>0  \text{ and}\\
    &\hspace{2cm}\text{if }A_\eps(r,q)=\ov\alpha=\alpha_1=\alpha_2\in
    [0,1]  \text{ then }
    \min\left\{\frac{1}{r},\frac{1}{q'}\right\}> \frac{k+2\ov\alpha}{2(k+1)} \Big\}.
  \end{align*}
  Since $A_\eps(r,q)\geq  \ov\alpha$ for some $\eps>0$ is equivalent to 
  $$
    \frac{1}{r}-\frac{1}{q}
    \geq \frac{2\ov\alpha}{k+2}\quad\text{and}\quad  
     \min\left\{\frac{1}{r},\frac{1}{q'}\right\} \,
     \begin{cases} 
       \,\geq \frac{k+2\ov\alpha}{2(k+1)} &\text{if }\ov\alpha<1 \text{ and }\alpha_1\neq \alpha_2 \\
       \,> \frac{k+2\ov\alpha}{2(k+1)} &\text{if }\ov\alpha=1 \text{ or }\alpha_1=\alpha_2.
     \end{cases},
  $$
  This proves the claim for $0<\kappa<1$. In the case $\kappa\in \{0,1\}$ the claim follows
  from \eqref{eq:A(0)} and Lemma~\ref{lem:SobolevIneq}~(i),(ii).
\end{proof}
 
 \noindent\textbf{Proof of Theorem~\ref{thm:GNhigherDspecial}:} We apply Theorem~\ref{thm:GNhigherD} to
 $P_1(D)=|D|^s-1, P_2(D)=I$. Again, the hypotheses of the Theorem hold for
 $(\alpha_1,\alpha_2,s_1,s_2,k)=(1,0,s,0,d-1)$ because $S$ is the unit sphere with $d-1$ non-vanishing
 principal curvatures.
\qed
 
%  and also incorporate estimates of the form 
% $$ 
%   \|u\|_2^2 
%   = \| \hat u\|_2^2
%   = \int_{\R^d} P_1(D)u\cdot P_2(D)u\,dx
%   \leq \|P_1(D)u\|_r \|P_2(D)u\|_{r'}
% $$
% whenever $P_1,P_2:=1/P_1$ satisfy assumption (A1) for $\alpha_1>0>\alpha_2:=-\alpha_1$.  

\section{Local Gagliardo-Nirenberg inequalities}

In~\cite{FerJeaManMar} it was shown that a ``local'' version of Gagliardo-Nirenberg inequalities is of
interest, too. Here one looks for a larger set of exponents where~\eqref{eq:GNgeneral} holds
under the additional hypothesis $\|P_1(D)u\|_{r_1}\leq R\|P_2(D)u\|_{r_2}$ where $R>0$ is fixed, see
Corollary~2.10 in that paper. A simple consequence of our estimates above is the following.
  
\begin{cor}\label{cor:LocalGN}
  Assume $d\in\N,\kappa\in [0,1]$ and (A1),(A2) for $s_1,s_2\in\R$ and $\alpha_1,\alpha_2>-1$. Then
  the inequality 
  $$
     \| u \|_q  \les (R^{\kappa-\kappa_1}+R^{\kappa-\kappa_2})\|P_1(D)u\|_{r_1}^{1-\kappa}
     \|P_2(D)u\|_{r_2}^\kappa 
  $$
  holds for all $u\in\mathcal S(\R^d)$ and satisfying $\|P_1(D)u\|_{r_1}\leq R\|P_2(D)u\|_{r_2}$
  provided that $(q,r_1,r_2)\in \mathcal A(\kappa_1)\cap \mathcal
   B(\kappa_2)$ holds for some $\kappa_1,\kappa_2\in [0,\kappa]$.
\end{cor}
\begin{proof}
  Choose $\kappa_1,\kappa_2$ as required. Then  Proposition~\ref{prop:CritFreqInterpolated}  gives
  \begin{align*}
    \|u_1\|_q  
    &\les\|P_1(D)u\|_{r_1}^{1-\kappa_1} \|P_2(D)u\|_{r_2}^{\kappa_1} \\
    &= (\|P_1(D)u\|_{r_1}\|P_2(D)u\|_{r_2}^{-1})^{\kappa-\kappa_1}  \cdot \|P_1(D)u\|_{r_1}^{1-\kappa}
    \|P_2(D)u\|_{r_2}^{\kappa} \\
    &\les R^{\kappa-\kappa_1} \|P_1(D)u\|_{r_1}^{1-\kappa} \|P_2(D)u\|_{r_2}^{\kappa}.
  \intertext{Similarly, Proposition~\ref{prop:LargeFreqInterpolated}   implies} \|u_2\|_q  
    &\les R^{\kappa-\kappa_2} \|P_1(D)u\|_{r_1}^{1-\kappa} \|P_2(D)u\|_{r_2}^{\kappa}.
  \end{align*}
  Summing up these inequalities gives the claim.
\end{proof}

In the context of our particular example $P_1(D)=|D|^s-1,s>0$ and $P_2(D)=I$ this gives the following
generalization of~\cite[Corollary~2.10]{FerJeaManMar}.

\begin{cor}\label{cor:LocalGN2}
  Assume $d\in\N,d\geq 2,\kappa\in (0,1),s>0$. Then 
  $$
     \|u\|_q \les  (R^{\kappa}+1) \|(|D|^s-1)u\|_r^{1-\kappa}\|u\|_r^\kappa 
  $$
  holds  for all $u\in\mathcal S(\R^d)$ satisfying $\|(|D|^s-1)u\|_r\leq R\|u\|_r$ provided
  that $(q,r)\neq (\infty,\frac{d}{s})$ if $0<s\leq d$  and 
  \begin{itemize}
    \item[(i)] $d=1,\, 1\leq r,q\leq \infty\;\text{and}\;  1-\kappa \leq \frac{1}{r}-\frac{1}{q} \leq s$ 
    \quad  or
    \item[(ii)] $d\geq 2, 1\leq r\leq 2\leq q\leq \infty\;\text{and}\;  \frac{2(1-\kappa)}{k+2} 
    \leq \frac{1}{r}-\frac{1}{q} 
    \leq  \frac{s}{d},\, \min\{\frac{1}{r},\frac{1}{q'}\} \geq
     \frac{k+2-2\kappa}{2(k+1)}$.
  \end{itemize}
\end{cor}
\begin{proof}
  This corresponds to the special case $(\kappa_1,\kappa_2)=(\kappa,0)$ and
  $(\alpha_1,\alpha_2,s_1,s_2,k,r_1,r_2)=(1,0,s,0,d-1,r,r)$ in
  Corollary~\ref{cor:LocalGN}. The computation of $\mathcal A(\kappa)$ and $\mathcal B(0)$  can be done as
  in the proof of Theorem~\ref{thm:GNhigherD}. Note that the assumptions imply $\ov\alpha=1-\kappa \in (0,1), \alpha_1\neq \alpha_2$ and
  $0<\kappa<1$.
\end{proof}

\section{Gagliardo-Nirenberg inequalities with unbounded characteristic sets} \label{sec:GN_wave}

In the previous sections we provided a systematic study of Gagliardo-Nirenberg Inequalities where the
characteristic set $S$ of the symbols is smooth and compact. In the case of unbounded characteristic sets our
analysis works for Schwartz functions whose Fourier transform is supported in some smooth and compact piece
of $S$, but an argument  for general Schwartz functions is lacking so far, even in
the case of simple differentiable operators with suitable scaling behaviour like the wave operator or the Schr\"odinger operator. 
In the $L^2$-setting, a less technical approach based on Plancherel's identity can be used. 
We follow the ideas presented in \cite{FerJeaManMar} to prove Gagliardo-Nirenberg inequalities of the form
\begin{align}\label{eq:GN_wave}
  \|u\|_q &\les \|\partial_{tt} u -\Delta u\|_r^{1-\kappa} \|u\|_r^\kappa \qquad (u\in\mathcal S(\R^d)),
  \\
  \|v\|_q &\les \|i\partial_t v - \Delta v\|_r^{1-\kappa} \|v\|_r^\kappa
  \qquad (v\in\mathcal S(\R^d)).
  \label{eq:GN_Schroe}
\end{align}
where $r=2$. We denote the space-time variable by $z=(x,t)\in\R^{d-1}\times\R=\R^d$.

\begin{thm}\label{thm:GN_wave}
 Let $d\in\N$. Then~\eqref{eq:GN_wave} holds provided that $r=2,q=\frac{2d}{d-4+4\kappa}$ 
 where  $\frac{1}{2}\leq\kappa\leq 1$ if $d\geq 3$
 and $\frac{1}{2}<\kappa\leq 1$ if $d=2$.
\end{thm}
\begin{proof}
  We first consider the case $d\geq 3$, define  $\mathcal C_t:= \{\xi=(\xi',\xi_d)\in\R^d:
  \xi_d^2-|\xi'|^2=t\}$ and the induced surface measure $\sigma_t$. Then we have the representation formula 
  \begin{align*}
    u(z) 
    = c_d \int_{\R^d} \hat u(\xi) e^{iz\cdot\xi}\,d\xi 
    = \frac{c_d}{2} \int_\R \int_{\mathcal C_t} \hat u(\xi)|\xi|^{-1}  e^{iz\cdot\xi}\,d\sigma_t(\xi)\,dt
  \end{align*}  
  where $c_d=(2\pi)^{-d/2}$. Strichartz' inequality from \cite{Strich} (Theorem I, case III (b)) implies that
  we have for $\frac{2(d+1)}{d-1}\leq q\leq \frac{2d}{d-2}$ 
  \begin{align*}
    \|u\|_q 
    &\les   \int_\R \|\mathcal F^{-1}\left(\hat u |\cdot|^{-1}\,d\sigma_t\right)\|_q \,dt \\
    &\les   \int_\R |t|^{\frac{d-1}{4}-\frac{d}{2q}} \|\hat u |\cdot|^{-1}\|_{L^2(\mathcal C_t,\,d\sigma_t)}
    \,dt
    \\
    &\les   \int_\R |t|^{\frac{d-2}{4}-\frac{d}{2q}} \|\hat u |\cdot|^{-1/2}\|_{L^2(\mathcal C_t,\,d\sigma_t)}
    \,dt.
  \end{align*}   
  Here, the factor $|t|^{\frac{d-1}{4}-\frac{d}{2q}}$ is obtained via scaling and in the last estimate we used 
  $|\xi|\geq \sqrt{|t|}$ for $\xi\in\mathcal C_t$.  
  On the other hand, Plancherel's Theorem gives
  \begin{align*}
    \|\partial_{tt}u-\Delta u\|_2^2 
    &= \int_{\R^d} |\xi_d^2-|\xi'|^2|^2|\hat u(\xi)|^2 \,d\xi \\
    &= \frac{1}{2} \int_\R \int_{\mathcal C_t}  |t|^2  |\hat u(\xi)|^2|\xi|^{-1} \,d\sigma_t(\xi)\,dt \\
    &= \frac{1}{2}\int_\R t^2  \|\hat u |\cdot|^{-1/2}\|_{L^2(\mathcal C_t,\,d\sigma_t)}^2 \,dt 
    \intertext{and } 
    \|u\|_2^2 
    &= \frac{1}{2}\int_\R  \|\hat u |\cdot|^{-1/2}\|_{L^2(\mathcal C_t,\,d\sigma_t)}^2 \,dt.
  \end{align*}   
  Writing $\varphi(t):= \|\hat u |\cdot|^{-1/2}\|_{L^2(\mathcal C_t,\,d\sigma_t)}$ it remains to
  prove that the quotient
  \begin{align*}
    \frac{\int_\R |t|^{\frac{d-2}{4}-\frac{d}{2q}} \varphi(t)\,dt}{
    (\int_\R t^{2}\varphi(t)^2\,dt)^{\frac{1-\kappa}{2}} 
    (\int_\R \varphi(t)^2\,dt)^{\frac{\kappa}{2}}
    } 
  \end{align*}
  is bounded independently of $\varphi$. According to \cite[Lemma~2.1]{FerJeaManMar}, with 
  $w(t)=|t|^{\frac{d-2}{4}-\frac{d}{2q}}$, $w_1(t)=1$ and $w_2(t)=t$, this is the case if and only if the
  following quantity is finite:
  \begin{align*} 
    \sup_{s>0} s^{\frac{1-\kappa}{2}}\left\|\frac{w}{(w_1^2+s w_2^2)^{1/2}}\right\|_{L^2(\R)}
    &= \sup_{s>0} s^{\frac{1-\kappa}{2}}
    \left( \int_\R \frac{|t|^{\frac{d-2}{2}-\frac{d}{q}}}{1+st^2}\,dt\right)^{\frac{1}{2}} \\
    &= \sup_{s>0} s^{\frac{1-\kappa}{2}-\frac{1}{4}(\frac{d}{2}-\frac{d}{q})}
    \left( \int_\R \frac{|\rho|^{\frac{d-2}{2}-\frac{d}{q}}}{1+\rho^2}\,d\rho\right)^{\frac{1}{2}}.
  \end{align*}
  This leads to $q=\frac{2d}{d-4+4\kappa}$. In view of 
  $\frac{2(d+1)}{d-1}\leq q\leq \frac{2d}{d-2}$ this requires $\frac{1}{2}\leq \kappa\leq \frac{d+2}{2(d+1)}$,
  but the upper bound for $\kappa$ may be removed just as in \cite[p.20-21]{FerJeaManMar} by combining the
  already established inequality for $\frac{2(d+1)}{d-1}$ with  
  $$
    \|u\|_q\leq \|u\|_2^{1-\theta}\|u\|_{\frac{2(d+1)}{d-1}}^\theta 
    \qquad 2\leq q\leq \frac{2(d+1)}{d-1},\quad
    \frac{1}{q}=\frac{1-\theta}{2}+\frac{\theta}{\frac{2(d+1)}{d-1}}. 
  $$  
  In the case $d=2$ the analogous reasoning based on Theorem I, case III (c) \cite{Strich}  
  shows that the above estimates are valid for 
  $6=\frac{2(d+1)}{d-1}\leq q< \frac{2d}{d-2}=\infty$ and thus $\frac{1}{2}<\kappa\leq 
  \frac{d+2}{2(d+1)}$. The same interpolation trick then allows to extend this to the whole range
  $\kappa>\frac{1}{2}$,
\end{proof}

We now apply this method to the Schr\"odinger operator.

\begin{thm}\label{thm:GN_Schroe}
 Let $d\in\N,d\geq 2$. Then~\eqref{eq:GN_Schroe} holds provided that $r=2,q=\frac{2(d+1)}{d-3+4\kappa}$ 
 and $\frac{1}{2}\leq\kappa\leq 1$.
\end{thm}
\begin{proof}
  Define  $\mathcal P_t:= \{\xi=(\xi',\xi_d)\in\R^d:  \xi_d-|\xi'|^2=t\}$ and the induced surface measure
  $\sigma_t$.  Plancherel's identity gives
\begin{align*}
  \|v\|_2^2
  &=  \int_\R \int_{\R^{d-1}} |\hat v(\xi',t+|\xi'|^2)|^2    \,d\xi' \,dt \\
  &=  \int_\R |t|^{\frac{d-1}{2}}\int_{\R^{d-1}} |\hat v(\sqrt t \xi',t(1+|\xi'|^2))|^2 
  \,d\xi' \,dt  \\
  &=  \int_\R |t|^{\frac{d-1}{2}} \int_{\R^{d-1}} |\hat v_t|^2 \sqrt{1+4|\xi'|^2} \,d\xi' \,dt \\
  &=  \int_\R |t|^{\frac{d-1}{2}} \|\hat v_t \|_{L^2(\mathcal P_1,\,d\sigma_1)}^2    \,dt         
\intertext{ where $\hat v_t(\xi):= \hat v(\sqrt t \xi',t\xi_d)  (1+4|\xi'|^2)^{-1/4}$.
  Similarly,}
  \|i\partial_t v-\Delta v\|_2^2
  &=  \int_\R  t^{2+\frac{d-1}{2}} \|\hat v_t \|_{L^2(\mathcal P_1,\,d\sigma_1)}^2    \,dt.         
\end{align*}
 Strichartz' inequality from \cite{Strich} (Theorem I, case I) implies for
  $q=\frac{2(d+1)}{d-1}$  
  \begin{align*}
    \|v\|_q 
    &= \left\| c_d \int_\R \int_{\R^{d-1}} \hat v(\xi',t+|\xi'|^2)  e^{iz\cdot
    (\xi',t+|\xi'|^2)}\,d\xi'\,dt\right\|_q  \\
    &\les  \int_\R \left\|   \int_{\R^{d-1}} \hat v(\xi',t+|\xi'|^2)  e^{iz\cdot (\xi',t+|\xi'|^2)}\,d\xi'
    \right\|_q \,dt \\
    &\les  \int_\R |t|^{\frac{d-1}{2}} \left\|   \int_{\R^{d-1}} \hat v(\sqrt t\xi',t(1+|\xi'|^2)) 
    e^{iz\cdot (\sqrt t\xi',t(1+|\xi'|^2))}\,d\xi' \right\|_q \,dt    \\
    &\les  \int_\R |t|^{\frac{d-1}{2}} \left\|  \mathcal F^{-1}\left( \hat v(\sqrt t\xi',t\xi_d) 
    (1+4|\xi'|^2)^{-1/2}\,d\sigma_1\right)(\sqrt tz',tz_1) \right\|_q \,dt \\
    &=  \int_\R |t|^{\frac{d-1}{2}-\frac{d+1}{2q}} \left\|  \mathcal F^{-1}\left( \hat v(\sqrt t\xi',t\xi_d)
    (1+4|\xi'|^2)^{-1/2}\,d\sigma_1\right)  \right\|_q \,dt \\
    &\les  \int_\R |t|^{\frac{d-1}{2}-\frac{d+1}{2q}} \left\| \hat v(\sqrt t\xi',t\xi_d)
    (1+4|\xi'|^2)^{-1/2}    \right\|_{L^2(\mathcal P_1,\,d\sigma_1)} \,dt \\
    &\les \int_\R |t|^{\frac{d-1}{2}-\frac{d+1}{2q}}  \|  \hat v_t\|_{L^2(\mathcal P_1,\,d\sigma_1)} \,dt. 
  \end{align*}    
  We set $\varphi(t):= |t|^{\frac{d-1}{4}}\|\hat v_t\|_{L^2(\mathcal P_1,\,d\sigma_1)}$ and it
  remains to show that the quotient
  \begin{align*}
    \frac{\int_\R |t|^{\frac{d-1}{4}-\frac{d+1}{2q}} \varphi(t)\,dt}{
    (\int_\R t^{2}\varphi(t)^2\,dt)^{\frac{1-\kappa}{2}} 
    (\int_\R \varphi(t)^2\,dt)^{\frac{\kappa}{2}}
    } 
  \end{align*}
  is bounded independently of $\varphi$. We apply \cite[Lemma~2.1]{FerJeaManMar} once more. 
  \begin{align*}
    \sup_{s>0} s^{\frac{1-\kappa}{2}} \left(\int_\R \frac{ |t|^{\frac{d-1}{2}-\frac{d+1}{q}}}{1+st^2}\,dt \right)^{\frac{1}{2}}
    &=   \sup_{s>0} s^{\frac{1-\kappa}{2}} \left( \left(\frac{1}{\sqrt s}\right)^{\frac{d+1}{2}-\frac{d+1}{q}} 
    \int_\R \frac{|\rho|^{\frac{d-1}{2}-\frac{d+1}{q}}}{1+\rho^2}\,d\rho \right)^{\frac{1}{2}} \\
    &=  \sup_{s>0} s^{\frac{1-\kappa}{2}-\frac{d+1}{8}+\frac{d+1}{4q}} 
    \left(\int_\R \frac{|\rho|^{\frac{d-1}{2}-\frac{d+1}{q}}}{1+\rho^2}\,d\rho \right)^{\frac{1}{2}}.  
  \end{align*} 
  This term is indeed finite for $q=\frac{2(d+1)}{d-1}$ and $\kappa=\frac{1}{2}$, which proves the claim in
  this special case. The claim for general $\kappa\geq \frac{1}{2}$ follows as above by interpolation.
\end{proof}

 We conjecture that  at least for $1< r\leq 2\leq q<\infty$ and $0<\kappa<1$ the
 inequality~\eqref{eq:GN_wave} actually holds for exponents 
 \begin{equation}\label{eq:GN_wave_general}
    \frac{1}{r}-\frac{1}{q}=\frac{2(1-\kappa)}{d},\qquad 
    \min\left\{\frac{1}{r},\frac{1}{q'}\right\} \geq \frac{d-2\kappa}{2(d-1)}
 \end{equation} 
  whereas the corresponding inequality involving the Schr\"odinger operator holds whenever  
  $$
    \frac{1}{r}-\frac{1}{q}=\frac{2(1-\kappa)}{d+1},\qquad 
    \min\left\{\frac{1}{r},\frac{1}{q'}\right\} \geq \frac{d+1-2\kappa}{2d}.
  $$   
  Note that the Sobolev inequalities \cite[Theorem~1.1]{JeongKwonLee_Uniform} then take the form of the
  endpoint estimate $\kappa=0$ in ~\eqref{eq:GN_wave_general}.

\section*{Acknowledgments}

The author thanks Robert Schippa and Louis Jeanjean  %~(Besan\c{c}on) 
for helpful comments related to this work.  Funded by the Deutsche Forschungsgemeinschaft (DFG, German Research
Foundation) -- Project-ID 258734477 -- SFB~1173.  
 
\bibliographystyle{abbrv}
\bibliography{biblio}

\end{document}